\documentclass[10pt]{amsart}

\usepackage{mathtools}
\usepackage[round]{natbib}
\setcitestyle{numbers,square}
\usepackage{hyperref}
\usepackage{xcolor}
\usepackage[normalem]{ulem}

\usepackage{accents}

\newcommand\redsout{\bgroup\markoverwith{\textcolor{red}{\rule[0.5ex]{2pt}{0.4pt}}}\ULon}

\newcommand{\E}{\mathbb{E}}
\newcommand{\N}{\mathbb{N}}
\newcommand{\Z}{\mathbb{Z}}
\newcommand{\Q}{\mathbb{Q}}
\newcommand{\R}{\mathbb{R}}

\newcommand{\Pb}{\mathbb{P}}
\newcommand{\tp}{t^{\prime}}
\newcommand{\pa}{a^{\prime}}
\newcommand{\kp}{k^{\prime}}

\newcommand{\ps}{s^{\prime}}

\newcommand{\hp}{h^{\prime}}

\newcommand{\ve}{\varepsilon}

\newcommand{\oun}{\overline{u}^{(n)}}

\newcommand{\uun}{\underline{u}^{(n)}}

\newcommand{\vep}{\varepsilon^{\prime}}

\newcommand{\odelta}{\overline{\delta}}

\def\={{\;\mathop{=}\limits^{\text{(law)}}\;}}

\newtheorem{theorem}{Theorem}[section]
\newtheorem{prop}[theorem]{Proposition}
\newtheorem{lemma}[theorem]{Lemma}
\newtheorem{defi}[theorem]{Definition}
\newtheorem{corollary}[theorem]{Corollary}
\newtheorem{hyp}[theorem]{Hypothesis} 

\newtheorem{remark}[theorem]{Remark}

\numberwithin{equation}{section}

\usepackage[english]{babel}

\DeclarePairedDelimiter\floor{\lfloor}{\rfloor}

\title[Path-by-path uniqueness of multidimensional SDE's on the plane]
{Path-by-path uniqueness of multidimensional SDE's on the plane with nondecreasing coefficients}

\author[A.-M. Bogso]{Antoine-Marie Bogso}
\address{University of Yaounde I\\
	Faculty of Sciences, Department of Mathematics\\
	P.O. Box 812, Yaounde, Cameroon \\
	 and African institute for Mathematical Sciences Ghana, P.O. Box LGDTD 20046, Summerhill Estates, East Legon Hills, Santoe, Acrra, Ghana}
\email{antoine.bogso@facsciences-uy1.cm, antoine@aims.edu.gh}           

\author[M. Dieye]{Moustapha Dieye}

\address{ \'Ecole polytechnique de Thi\`es, D\'epartement tronc commun, BP A10, Thi\`es, S\'en\'egal}

\email{moustapha@aims.edu.gh}

\author[O. Menoukeu Pamen]{Olivier Menoukeu Pamen}
\address{Institute for Financial and Actuarial Mathematics (IFAM) \\
	Department of Mathematical Sciences, University of Liverpool \\
	Liverpool L69 7ZL, UK \\
	and African institute for Mathematical Sciences Ghana, P.O. Box LGDTD 20046, Summerhill Estates, East Legon Hills, Santoe, Acrra, Ghana}
\email{menoukeu@liverpool.ac.uk}
\thanks{The project on which this publication is based has been carried out with funding provided by the Alexander von Humboldt Foundation, under the programme financed by the German Federal Ministry of Education and Research entitled German Research Chair No 01DG15010.}


\date{\today}

\begin{document}

 \begin{abstract}
 In this paper we study path-by-path uniqueness for multidimensional stochastic differential equations driven by the Brownian sheet. We assume that the drift coefficient is unbounded, verifies a spatial linear growth condition and is componentwise nondeacreasing. Our approach consists  of showing the result for bounded and componentwise nondecreasing drift using both a local time-space representation and a law of iterated logarithm for Brownian sheets. The desired result follows using a Gronwall type  lemma on the plane.  As a by product, we obtain the existence of a unique strong solution of multidimensional SDEs driven by the Brownian sheet when the drift is non-decreasing and satisfies a spatial linear growth condition.  
 \end{abstract}

\maketitle 

\section{Introduction}
\label{intro}
In this work, we consider the following system of stochastic integral equations on the plane with additive noise:
\begin{align}\label{eqmainre1}
	X_{s,t}-X_{s,0}-X_{0,t}+X_{0,0}=\int_0^t\int_0^sb(\xi,\zeta,X_{\xi,\zeta})\mathrm{d}\xi\mathrm{d}\zeta+W_{s,t} 
	\text{ for }(s,t)\in\R_+^2,
\end{align}
where $b:\R_+^2\times \mathbb{R}^d\to\R^d$ is Borel measurable satisfying some conditions that will be specified later and $W=(W_{s,t},(s,t)\in\R_+^2)$ is a $d$-dimensional Brownian sheet given on a filtered probability space $(\Omega,\mathcal{F},(\mathcal{F}_{s,t},(s,t)\in\R_+^2),\Pb)$ with $\partial W=0$, where $\partial W$ stands for the restriction of $W$ to the  boundary $\partial D=\{0\}\times\R_+\cup\R_+\times\{0\}$ of $D=\R_+^2$. We endow $D$ with the partial order ``$\preceq$" (respectively ``$\prec$") defined by
$$
(s,t)\preceq(\ps,\tp)\text{ when }s\leq \ps\text{ and }t\leq\tp,
$$
respectively
$$
(s,t)\prec(\ps,\tp)\text{ when } s<\ps\text{ and }t<\tp.
$$
Observe that \eqref{eqmainre1} is a particular case of the more general non-Markovian type equation 
\begin{align}\label{eqmainre2}
	&	X_{s,t}-X_{s,0}-X_{0,t}+X_{0,0}\nonumber\\
	&=\int_0^t\int_0^sb(\xi,\zeta,X_{\xi,\zeta})\mathrm{d}\xi\,\mathrm{d}\zeta+\int_0^t\int_0^sa(\xi,\zeta,X_{\xi,\zeta})\,\mathrm{d}W_{\xi,\zeta} 
	\text{ for }(s,t)\in\R_+^2,
\end{align}
where $a:\,\R_+^2\times\R^d\to\R^d\times\R^d$ is a Borel measurable matrix function. 
Note that \eqref{eqmainre2} appears as an integral equation when one rewrites the following quasilinear stochastic hyperbolic differential equation 
\begin{align}\label{eqmainre2b}
	\frac{\partial^2 X_{s,t}}{\partial s\partial  t}=b(s,t,X_{s,t})+a(s,t,X_{s,t})\frac{\partial^2 W_{s,t}}{\partial s\partial t},
\end{align}
where the notation ``$\frac{\partial^2 W_{s,t}}{\partial s\partial t}$'' designates a white noise on $D$. As pointed out by Farr\'e and Nualart \cite{FaNu93} (see also \cite{QuTi07}), a formal $\frac{\pi}{4}$ rotation transforms \eqref{eqmainre2b} into a nonlinear stochastic wave equation.  This idea, thanks to Walsh \cite{Wa78}, has been used by Carmona and Nualart \cite{CaNu88} to provide existence and uniqueness results for the 1-dimension stochastic wave equation
\begin{align}\label{eqmainre2c}
	\frac{\partial^2 Y}{\partial t^2}(t,x)-\frac{\partial^2 Y}{\partial x^2}(t,x)=a(Y(t,x))\dot{W}(t,x)+b(Y(t,x))
\end{align}  
with some initial conditions $Y(0,\cdot)$ and $\dfrac{\partial Y}{\partial t}(0,\cdot)$, where $t$ varies in $\R_+$, $x$ varies in $\R$ and $\dot{W}$ denotes a white noise in time as well as in space. Reformulation of \eqref{eqmainre2c} using a $\frac{\pi}{4}$ rotation allows use of the rectangular increments of both $t$ and $x$ (see e.g. \cite[Section 1]{QuTi07}).

The problem \eqref{eqmainre1} can also be interpreted as a noisy analog of the so-called Darboux problem given by
\begin{align}\label{eqdarboux1}
	\frac{\partial^2y}{\partial s\partial t}=b\Big(s,t,y,\frac{\partial y}{\partial s},\frac{\partial y}{\partial t}\Big) \quad\text{for }(s,t)\in[0,T]\times[0,T],
\end{align}
with the initial conditions 
\begin{align}\label{eqdarboux2}
	y(0,t)=\sigma(t)\,\text{ on }[0,T]\,\text{ and }\,y(s,0)=\tau(s)\,\text{ on }[0,T],
\end{align}
where $\sigma$ and $\tau$ are absolutely continuous on $[0,T]$.
Using Caratheodory's theory of differential equations, Deimling \cite{De70} proved an existence theorem for the system \eqref{eqdarboux1}-\eqref{eqdarboux2} when $b$ is Borel measurable in the first two variables and bounded and continuous in the last three variables.

Existence and uniqueness of solutions to stochastic differential equations (SDEs) driven by a Brownian sheet has been widely studied. In the time homogeneous case, Cairoli \cite{Ca72} proved that \eqref{eqmainre2} has a unique strong solution when the coefficients are Lipschitz continuous. This result was generalised to the time dependent coefficients by Yeh \cite{Ye81} under an additional growth condition. Weak existence of solutions to \eqref{eqmainre2} was derived in  \cite{Ye85} assuming that the coefficients are continuous, satisfy a growth condition and the initial value has moment of order six. In all of the above mentioned works, the coefficients are at least continuous. Nualart and Tindel \cite{NuTi97}, show that \eqref{eqmainre2} has a unique strong solution when the drift is nondecreasing and bounded. Their results were extended to SDEs driven by two parameter martingales in \cite{NuTi98} (see also \cite{CaNu88,FaNu93} for further extensions). In \cite{ENO03}, the authors generalised the above results to SDEs driven by a fractional Brownian sheet.

In this work we are concerned with a different uniqueness question. In particular, we look at the notion of path-by-path uniqueness introduced by Davie \cite{Da07} (see also Flandoli \cite{Fla10}). Let $\mathcal{V}$, resp. $\partial \mathcal{V}$ be the space of continuous $\R^d$-valued functions on $D$, resp. $\partial D$. The following definition can be seen as a counterpart of \cite[Definition 1.5]{Fla10} in the case of two parameter processes.
\begin{defi}\label{defipathbpath}
	We say that the path-by-path uniqueness of solutions to \eqref{eqmainre1} holds when there exists a full $\Pb$-measure set $\Omega_0\subset\Omega$ such that for all $\omega\in\Omega_0$ the following statement is true: there exists at most one function  $y\in\mathcal{V}$ which satisfies $$\int_0^T\int_0^T|b(\xi,\zeta,y_{\xi,\zeta})|\mathrm{d}\xi \mathrm{d}\zeta<\infty,\text{ }\partial y=x,\text{ for some }x\in\partial \mathcal{V}\text{ and }T>0$$ and
	\begin{align}\label{eqmainpathbypath}
		y_{s,t}=x+\int_0^s\int_0^tb(\xi,\zeta,y_{\xi,\zeta})\mathrm{d}\xi \mathrm{d}\zeta+W_{s,t}(\omega),\text{ }\forall\,(s,t)\in[0,T]^2.
	\end{align}
\end{defi}
One of the motivations for studying path-by-path uniqueness comes from the regularisation by noise of random ODEs. For instance, let $v$ be a continuous function and let us consider the following one parameter equation in $\mathbb{ R}^d$
$$
X_t=X_0+\int_0^tb (s, X_s)\mathrm{d}s +v_t.
$$
We know that there exists a unique solution to the above equation when $b$  is Lipschitz in $x$, uniformly in $t$, with uniform linear growth. Observe that uniqueness also holds when $b$ is only locally Lipschitz. Under some weak conditions on $b$ the corresponding equation without $v$ might be ill-posed or uniqueness could not be valid. For example when $b$ is merely bounded and measurable one may ask if there is a notion of uniqueness if $v$ has some specific features. In other words, can we find a path $v$ that regularises the equation? The result obtained in \cite{Da07} shows that when $b\in L^\infty$, the Brownian path regularises the drift $b$ in the sense of Definition \ref{defipathbpath}. 
In addition, as shown in \cite[Section 1.8.5]{BFGM14}, path-by-path uniqueness is much stronger than pathwise uniqueness. Indeed Shaposhnikov and Wresch \cite[Section 4]{ShWr20} exhibit SDEs such that strong solutions exist, pathwise uniqueness holds and path-by-path uniqueness fails to hold. This is another motivation for studying path-by-path uniqueness even when pathwise uniqueness holds.

In the case of Brownian motion, when the drift is bounded and measurable, and the diffusion is reduced to the identity, the path-by-path uniqueness of equation \eqref{eqmainre2} was proved by Davie in \cite{Da07}. This result was extended in several directions. For non-constant diffusion, Davie in \cite{DavieM} proved path-by-path uniqueness of solution to \eqref{eqmainre2}, interpreting the equation in the rough path sense. In \cite{CG16}, the authors showed that path-by-path uniqueness holds if the Brownian motion is replaced by a $d$-dimensional fractional Brownian motion of Hurst parameter $H\in(0,1)$. It is also assumed that the drift $b$ can be merely a distribution as long as $H$ is sufficiently small. In \cite{Pri18}, Priola considered equations driven by a L\'evy process assuming that the drift is H\"older continuous (see \cite{Pri19} for the non-constant diffusion coefficient case). 

Path-by-path uniqueness is closely related to the regularisation by noise problem for ordinary (or partial) differential equations (ODEs or PDEs) which has recently drawn a lot of attention. Beck, Flandoli, Gubinelli and Maurelli \cite{BFGM14} proved a Sobolev regularity of solutions to the linear stochastic transport and continuity equations with drift in critical $L^p$ spaces. Such a result does not hold for the corresponding deterministic equations.  Butkovsky and Mytnik \cite{BM16} analysed the regularisation by noise phenomenon for a non-Lipschitz stochastic heat equation and proved path-by-path uniqueness for any initial condition in a certain class of a set of probability one. Amine, Mansouri and Proske \cite{AMP20} investigated path-by-path uniqueness for transport equations driven by the fractional Brownian motion of Hurst index $H<1/2$ with bounded vector-fields. In \cite{CG16,GG21} the authors solved the regularisation by noise problem from the point of view of additive perturbations. In particular, Galeati and Gubinelli \cite{CG16} considered generic perturbations without any specific probabilistic setting. Amine, Ba{\~n}os and Proske \cite{ABP17a} constructed a new Gaussian noise of fractional nature and proved that it has a strong regularising effect on a large class of ODEs. More recently, Harang and Perkowski \cite{HP21} studied the regularisation by noise problem for ODEs with vector fields given by Schwartz distributions and proved that if one perturbs such an equation by adding an infinitely regularising path, then it has a unique solution. Kremp and Perkowski \cite{KP20a} looked at multidimensional SDEs with distributional drift driven by symmetric $\alpha$-stable L\'evy processes for $\alpha\in(1,2]$. In all of the above mentioned works, the driving noise considered are one parameter processes. 

In what follows, we make use of the Girsanov theorem to show that the path-by-path uniqueness in our setting is equivalent to the uniqueness of a random ODE on the plane. For any $x\in\partial \mathcal{V}$ and any $\omega$ such that the  path $(s,t)\longmapsto W_{s,t}$ is continuous, we denote by $S(x,\omega)$ the set of functions in $\mathcal{V}$ that solve \eqref{eqmainpathbypath}. Under linear growth and monotonicity conditions on $b$, we prove that $S(x,\omega)$ has at most one element. By a vector translation argument, it suffices to show that $S(0,\omega)$ has no more than one element.

As in \cite[Section 1]{Da07}, we show the path-by-path uniqueness on $D^1=[0,1]^2$. Precisely, we consider the integral equation
\begin{align}\label{eqmainre1b}
	X_{s,t}-X_{s,0}-X_{0,t}+X_{0,0}=\int_0^t\int_0^sb(\xi,\zeta,X_{\xi,\zeta})\mathrm{d}\xi\mathrm{d}\zeta+W_{s,t}
	\text{ for }(s,t)\in D^1,
\end{align}
where the drift is of spatial linear growth. 
 There is no loss of generality in reducing the problem to $D^1$ since we can repeat the argument on any square $[m,m+1]\times[\ell,\ell+1]$, $(m,\ell)\in\N^2$, $m>0$.
We first suppose that $b$ is bounded and monotone.  Let $\mathcal{V}^1$ be the space of continuous $\R^d$-valued functions on $D^1$ and let $\mathcal{V}^1_0$ be the space of functions $y\in\mathcal{V}^1$ with $\partial y=0$, where $\partial y$ is the restriction of $y$ to $\partial D^1$ ($\partial D^1=\{0\}\times[0,1]\cup[0,1]\times\{0\}$). Let  $\Pb$ be the law of an $\R^d$-valued Brownian sheet on $D^1$ which vanishes on $\partial D^1$.	
	
 The function $L$ given by
\begin{align*}
	L(y)=\exp\Big(\int_0^1\int_0^1b(\xi,\zeta,y_{\xi,\zeta}) \mathrm{d}y_{\xi,\zeta}-\frac{1}{2}\int_0^1\int_0^1|b(\xi,\zeta,y_{\xi,\zeta})|^2\mathrm{d}\xi \mathrm{d}\zeta\Big)
\end{align*}
is well-defined for $\Pb$-a.e.  $y\in \mathcal{V}^1_0$. Moreover, if $y\in \mathcal{V}^1_0$ is chosen random, with law $\mathrm{d}\widetilde{\Pb}=L\mathrm{d}\Pb$, then, by Girsanov theorem (see for example \cite{Ca72, DM15, Im88}), the path $W$ defined by
\begin{align}\label{eqPathGirsanov}
	W_{s,t}=y_{s,t}-\int_0^s\int_0^tb(\xi,\zeta,y_{\xi,\zeta})\mathrm{d}\xi \mathrm{d}\zeta
\end{align}
has law $\Pb$. This means that $y$ is a solution to \eqref{eqmainre1b} with $W$ defined by \eqref{eqPathGirsanov}. Path-by-path uniqueness of solutions to \eqref{eqmainre1b} holds if and only if for $\widetilde{\Pb}$-a.e. $y\in \mathcal{V}^1_0$, $z=y$ is the only solution to
	\begin{align*}
		W_{s,t}=z_{s,t}-\int_0^s\int_0^tb(\xi,\zeta,z_{\xi,\zeta})\mathrm{d}\xi \mathrm{d}\zeta
	\end{align*} 
with $W$ given by \eqref{eqPathGirsanov}, which is equivalent to saying that for $\widetilde{\Pb}$-a.e. $y\in \mathcal{V}^1_0$, the only solution to
\begin{align}\label{eqOnPathbyPath}
	u(s,t)=\int_0^s\int_0^t\{b(\xi,\zeta,y_{\xi,\zeta}+u(\xi,\zeta))-b(\xi,\zeta,y_{\xi,\zeta})\}\mathrm{d}\xi \mathrm{d}\zeta
\end{align}
is $u=0$ (see e.g. \cite[Section 1]{Da07}). Since $\widetilde{\Pb}$ is absolutely continuous with respect to $\Pb$, it is enough to show that, if $W$ is an $\R^d$-valued Brownian sheet, then, with probability one, there is no nontrivial solution $u\in \mathcal{V}^1_0$ of 
\begin{align*}
	u(s,t)=\int_0^s\int_0^t\{b(\xi,\zeta,W_{\xi,\zeta}+u(\xi,\zeta))-b(\xi,\zeta,W_{\xi,\zeta})\}\mathrm{d}\xi \mathrm{d}\zeta.
\end{align*} 
This is the statement of Theorem \ref{maintheuniq2} which is extended to unbounded monotone drifts in Theorem \ref{maintheuniq1}. Our proof of Theorem \ref{maintheuniq2} relies on some estimates for an averaging  operator along the sheet (see Lemma \ref{lem:PseudoMetric1}). This result plays a key role in the proof of a Gronwall type lemma (see Lemma \ref{lem:GronwallSheetd}) which enables us  to prove path-by-path uniqueness of solutions to \eqref{eqmainre1}.  

The Yamada-Watanabe principle for one dimensional SDEs driven by Brownian sheets was derived in \cite{NuYe89} (see also \cite{Ye87}). More precisely, the authors show that combining weak existence and pathwise uniqueness yields existence of a unique strong solution in the two parameter setting. This result can be extended to the multidimensional case (see e.g. \cite[Remark 2]{Tu83}). When $b$ is of linear growth, we can show (see Lemma \ref{theqweak1}) that the SDE \eqref{eqmainre2} has a weak solution. The latter together with path by path uniqueness (and thus pathwise uniqueness) implies the existence of a unique strong solution to the SDE \eqref{eqmainre2} and therefore generalises some results in \cite{ENO03,NuTi97} to the multidimensional case. To the best of our knowledge, such a result has not been derived in the multidimensional case.


The remainder of the paper is structured as follows. In Section \ref{defcon}, we recall some basic definitions and concepts. The main results are stated and proved in Section \ref{sectmasinres}. In section \ref{prelresul}, we prove some preliminary results whereas Section \ref {Auxrel} is devoted to the proof of a number of auxiliary results.

\section{Basic definitions and concepts}\label{defcon}
In this section we recall some basic definitions and concepts for SDEs on the plane. We start with the definitions of filtered probability space and $d$-dimensional Brownian sheet that can be found in \cite{NuYe89,Ye81}.
\begin{defi}
	We call a filtered probability space any probability space $(\Omega,\mathcal{F},\Pb)$ with a family  $(\mathcal{F}_{s,t},(s,t)\in D)$ of sub-$\sigma$-algebras of $\mathcal{F}$ such that
	\begin{enumerate}
		\item $\mathcal{F}_{0,0}$ contains all null sets in $(\Omega,\mathcal{F},\Pb)$,
		\item $\{\mathcal{F}_{s,t},(s,t)\in D\}$ is nondecreasing in the sense that $\mathcal{F}_{s,t}\subset\mathcal{F}_{\ps,\tp}$ when $(s,t)\prec(\ps,\tp)$,
		\item $(\mathcal{F}_{s,t},(s,t)\in D)$ is a right-continuous system in the sense that
		$$
		\mathcal{F}_{s,t}=\bigcap\limits_{(\ps,\tp)\prec(s,t)}\mathcal{F}_{\ps,\tp}.
		$$
	\end{enumerate}
\end{defi}

\begin{defi}
	We call a one-dimensional $(\mathcal{F}_{s,t})$-Brownian sheet on a filtered probability space $(\Omega,\mathcal{F},(\mathcal{F}_{s,t},(s,t)\in D),\Pb)$ any real valued two-parameter stochastic process $W^{(0)}=(W^{(0)}_{s,t},(s,t)\in D)$ satisfying the following conditions:
	\begin{enumerate}
		\item $W^{(0)}$ is $(\mathcal{F}_{s,t},(s,t)\in D)$-adapted, i.e. $W^{(0)}_{s,t}$ is $\mathcal{F}_{s,t}$-measurable for every $(s,t)\in D$.
		\item Every sample function $(s,t)\longmapsto W^{(0)}_{s,t}(\omega)$ of $W^{(0)}$ is continuous on $D$.
		\item For every finite rectangle of the type $\Pi=]s,\ps]\times]t,\tp]\subset D$, the random variable $$W^{(0)}(\Pi):=W^{(0)}_{\ps,\tp}-W^{(0)}_{s,\tp}-W^{(0)}_{\ps,t}+W^{(0)}_{s,t}$$
		is centered, Gaussian with variance $(\ps-s)(\tp-t)$ and independent of $\mathcal{F}_{s,1}\vee\mathcal{F}_{1,t}$.
	\end{enumerate} 
	We call a $d$-dimensional Brownian sheet any $\R^d$-valued two-parameter process $W=(W^{(1)},\ldots,W^{(d)})$ such that $W^{(i)}$, $i=1,\ldots,d$ are independent one-dimensional Brownian sheets.
\end{defi}

In the following, we discuss the notions of weak and strong solutions to the SDE \eqref{eqmainre2} (see for example \cite[Section 2]{Ye81}).  We start with the definition of a weak solution.  
\begin{defi}
	A weak solution to the SDE \eqref{eqmainre2} is a system $(\Omega,\mathcal{F},(\mathcal{F}_{s,t}),W,X=(X_{s,t}),\Pb)$ such that 
	\begin{enumerate}
		\item $(\Omega,\mathcal{F},(\mathcal{F}_{s,t},(s,t)\in D),\Pb)$ is a filtered probability space,
		\item $W=(W_{s,t},(s,t)\in D)$ is a $d$-dimensional $(\mathcal{F}_{s,t})$-Brownian sheet with $\partial W=0$,
		\item $X$ is $(\mathcal{F}_{s,t})$-adapted, has continuous sample paths and, $\Pb$-a.s.,
		\begin{align*}
			&	X_{s,t}-X_{s,0}-X_{0,t}+X_{0,0}\\
			&=\int_0^t\int_0^sb(\xi,\zeta,X_{\xi,\zeta})\mathrm{d}\xi\mathrm{d}\zeta+\int_0^t\int_0^sa(\xi,\zeta,X_{\xi,\zeta})\,\mathrm{d}W_{\xi,\zeta},\quad\forall\,(s,t)\in D.
		\end{align*}
	\end{enumerate}
\end{defi}
 
\begin{remark}
	When the drift $b$ satisfies a linear growth condition, weak existence holds for \eqref{eqmainre1b} (see Theorem \ref{theqweak1}).
\end{remark}

We now turn to the notion of strong solution. Let $\mathcal{B}(\mathcal{V})$ (respectively $\mathcal{B}(\partial \mathcal{V})$) be the $\sigma$-algebra of Borel sets in the space $\mathcal{V}$ (respectively $\partial \mathcal{V}$) of all continuous $\R^d$-valued functions on $D$ (respectively $\partial D$) with respect to the metric topology of uniform convergence on compact subsets of $D$. The subsequent definitions are borrowed from \cite{NuYe89}.
\begin{defi}
	Let $\overline{\mathcal{B}(\mathcal{V})}$ be the completion of $\mathcal{B}(\mathcal{V})$ with respect to the Wiener measure $m_\mathcal{V}$ on $(\mathcal{V},\mathcal{B}(\mathcal{V}))$ concentrated on $\mathcal{V}_0$. For every $(s,t)\in D$, we denote by $\mathcal{B}_{s,t}(\mathcal{V})$ the $\sigma$-algebra generated by the cylinder sets of the type $\{w\in \mathcal{V};\,w(\xi,\zeta)\in E\}$ for some $(\xi,\zeta)\preceq(s,t)$ and $E\in\mathcal{B}(R^d)$ and by $\overline{\mathcal{B}_{s,t}(\mathcal{V})}$ the $\sigma$-algebra generated by $\mathcal{B}_{s,t}(\mathcal{V})$ and all the null sets in $(\mathcal{V},\overline{\mathcal{B}(\mathcal{V})},m_\mathcal{V})$. Let $\overline{\mathcal{B}(\partial \mathcal{V}\times \mathcal{V})}^{\lambda\times m_\mathcal{V}}$ be the completion of $\mathcal{B}(\partial \mathcal{V}\times \mathcal{V})$ with respect to the product measure $\lambda\times m_\mathcal{V}$ for any probability measure $\lambda$ on $\partial \mathcal{V}$.
\end{defi}
\begin{defi}\label{DefiClassTransf}
	Let $\mathbf{T}(\partial \mathcal{V}\times \mathcal{V})$ be the class of transformations $F$ of $\partial \mathcal{V}\times \mathcal{V}$ into $\mathcal{V}$ which satisfies the condition that for every probability measure $\lambda$ on $(\partial \mathcal{V},\mathcal{B}(\partial \mathcal{V}))$, there exists a transformation $F_{\lambda}$ of $\partial \mathcal{V}\times \mathcal{V}$ into $\mathcal{V}$ such that
	\begin{enumerate}
		\item $F_{\lambda}$ is $\overline{\mathcal{B}(\partial \mathcal{V}\times \mathcal{V})}^{\lambda\times m_\mathcal{V}}/\mathcal{B}(\mathcal{V})$ measurable,
		\item For every $x\in\partial \mathcal{V}$, $F_{\lambda}[x,\cdot]$ is $\overline{\mathcal{B}_{s,t}(\mathcal{V})}/\mathcal{B}_{s,t}(\mathcal{V})$ measurable, for every $(s,t)\in D$,
		\item There exists a null set $N_{\lambda}$ in $(\partial \mathcal{V},\mathcal{B}(\partial \mathcal{V}),\lambda)$ such that $F[x,w]=F_{\lambda}[x,w]$ for almost all $w$ in $(\mathcal{V},\overline{\mathcal{B}(\mathcal{V})},m_\mathcal{V})$ and all $x\in \partial \mathcal{V}\setminus N_{\lambda}$.
	\end{enumerate}
\end{defi}

\begin{defi}
	Let $(X,W)$ be a weak solution to the SDE \eqref{eqmainre2} on a filtered probability space $(\Omega,\mathcal{F},\{\mathcal{F}_{s,t},(s,t)\in D\},\Pb)$ and let $\lambda$ be the probability distribution of $\partial X$. We call $(X,W)$ a strong solution to \eqref{eqmainre2} if there exists a transformation $F_{\lambda}$ of $\partial \mathcal{V}\times \mathcal{V}$ into $\mathcal{V}$ satisfying Conditions 1 and 2 of Definition \ref{DefiClassTransf} such that
	\begin{align*}
		X=F_{\lambda}[\partial X,W]\,\text{ }\Pb\text{-a.s. on }\Omega.
	\end{align*}
\end{defi}
Here is a well known concept of uniqueness associated to strong solutions of \eqref{eqmainre2} provided such solutions exist.
\begin{defi}
	We say that the SDE \eqref{eqmainre2} has a unique strong solution if there exists $F\in\mathbf{T}(\partial \mathcal{V}\times \mathcal{V})$ such that,
	\begin{enumerate}
		\item if $(\Omega,\mathcal{F},(\mathcal{F}_{s,t},(s,t)\in D),\Pb)$ is a filtered probability space on which an $\R^d$-valued $(\mathcal{F}_{s,t},(s,t)\in D)$-Brownian sheet $W$ with $\partial W=0$ exists, then for every continuous $(\mathcal{F}_{s,t},(s,t)\in D)$-adapted boundary process $Z$ on  $(\Omega,\mathcal{F},(\mathcal{F}_{s,t},(s,t)\in D),\Pb)$ whose probability distribution is denoted by $\lambda$, $(X,W)$ with $X=F(Z,W)$ is a weak solution of \eqref{eqmainre2} with $\partial X=Z$ $\Pb$-a.s. on $\Omega$.
		\item if $(X,W)$ is a weak solution of \eqref{eqmainre2} on a filtered probability space $(\Omega,\mathcal{F},(\mathcal{F}_{s,t},(s,t)\in D),\Pb)$ and the probability distribution of $\partial X$ is denoted by $\lambda$, then $X=F_{\lambda}[\partial X,W]$ $\Pb$-a.s. on $\Omega$.
	\end{enumerate}
\end{defi}
There are two classical notions of uniqueness associated to weak solutions (see e.g. \cite[Definitions 1.2 and 1.7]{NuYe89}). 
\begin{defi}
	We say that the solution to the SDE \eqref{eqmainre2} is unique in the sense of probability distribution if whenever $(X,W)$ and $(X^{\prime},W^{\prime})$ are two solutions of \eqref{eqmainre2} on two possibly different  filtered probability spaces and $\partial X=x=\partial X^{\prime}$ for some $x\in\partial \mathcal{V}$, then $X$ and $X^{\prime}$ have the same probability distribution on $(\mathcal{V},\mathcal{B}(\mathcal{V}))$.
\end{defi}
\begin{defi}
	We say that the pathwise uniqueness of solutions to the SDE \eqref{eqmainre2} holds if whenever $(X,W)$ and $(X^{\prime},W)$ with the same $W$ are two solutions to \eqref{eqmainre2} on the same probability space and $\partial X=\partial X^{\prime}$, then $X=X^{\prime}$ for $\Pb$-a.e. $\omega\in\Omega$.
\end{defi}

\section{Main results}\label{sectmasinres}

In this section, we present the main results of this paper. We assume the following conditions on the drift: 
\begin{hyp}\label{hyp1}\leavevmode
	\begin{enumerate}
		\item $b:[0,1]^2\times \mathbb{R}^d\rightarrow \mathbb{R}^d$ is a Borel measurable function satisfying the spatial linear growth condition, that is, there exists a constant $M$ such that 
		$$
		|b(t,s,x)|\leq M(1+|x|) \text{ for all } x\in \mathbb{R}^d.
		$$
		\item $b$ is componentwise nondecreasing in space, that is, each component $(b_i)_{1\leq i\leq d}$ is componentwise nondecreasing in space. More precisely for $x,y \in \mathbb{R}^d$, we have:
		$$
		x\preceq  y \Rightarrow b_i(x)\leq b_i(y), 1\leq i\leq d,
		$$
		where $x\preceq  y$ means $x_i\leq y_i$ for all $i\in\{1,\ldots,d\}$. 
	\end{enumerate}
\end{hyp}

\begin{theorem}\label{maintheuniq1}
	Suppose $b$ satisfies Hypothesis \ref{hyp1}. Then for almost every Brownian sheet path $W$, there exists a unique continuous function $X :[0,1]^2\rightarrow \mathbb{R}^d$ satisfying \eqref{eqmainre1b}.
\end{theorem}

 \begin{corollary}
		Suppose  $b$ is as in Theorem \ref{maintheuniq1}. Then the SDE \eqref{eqmainre1b} admits a unique strong solution. 
	\end{corollary}
\begin{proof}
	It follows from the fact that under the conditions of the Corollary, \eqref{eqmainre1b} has a weak solution. In addition, since path-by-path uniqueness implies pathwise uniqueness (see \cite[Page 9, Section 1.8.4]{BFGM14}), the result follows from the Yamada-Watanabe type principle for SDEs driven by Brownian sheets (see e.g. Nualart and Yeh \cite{NuYe89}).
	\end{proof}

The proof of Theorem \ref{maintheuniq1} relies on the following theorem

\begin{theorem}\label{maintheuniq2}
	Suppose $b$ is as in Theorem \ref{maintheuniq1}. Suppose in addition that $b$ is uniformly bounded. Then for almost every Brownian sheet path $W$, there exists a unique continuous function $X :[0,1]^2\rightarrow \mathbb{R}^d$ satisfying \eqref{eqmainre1b}.
\end{theorem}
Recall that using the Girsanov theorem, path-by-path uniqueness holds if there exists $\Omega_1\subset\Omega$ with $\Pb(\Omega_1)=1$ such that for any $\omega\in\Omega_1$, there is no nontrivial solution $u \in  C([0,1]^2,\mathbb{R}^d)$ to the following system of integral equations \begin{align}\label{eq:IntPathByPath}
	u(s,t)=\int_0^t\int_0^s\{b(\xi,\zeta,W_{\xi,\zeta}(\omega)+u(\xi,\zeta))-b(\xi,\zeta,W_{\xi,\zeta}(\omega))\}\mathrm{d}\xi\mathrm{d}\zeta,\text{ }\forall\,(s,t)\in[0,1]^2.
\end{align}
Let us also consider the set $\mathbf{Q}=[-1,1]^d$ and its dyadic decomposition. Recall that $x\in\mathbf{Q}$ is called a dyadic number if it is a rational with denominator a power of $2$. 
  The next theorem is equivalent to Theorem \ref{maintheuniq2}.
\begin{theorem}\label{theo:DavieSheetMonotoneD}
	Let $W:=\left(W_{s,t},(s,t)\in[0,1]^2\right)$ be a $d$-dimensional Brownian sheet defined on a filtered probability space $(\Omega,\mathcal{F},\mathbb{F},\Pb)$, where $\mathbb{F}=(\mathcal{F}_{s,t};s,t\in[0,1])$.
	Let $b$ be as in Theorem \ref{maintheuniq2}.
	Then there exists $\Omega_1\subset\Omega$ with $\Pb(\Omega_1)=1$ such that for any $\omega\in\Omega_1$, $u=0$ is the unique solution in $\mathcal{V}_0$ to the system of integral equations \eqref{eq:IntPathByPath}.
\end{theorem}

The proof of Theorem \ref{theo:DavieSheetMonotoneD} is carried out in two main steps. In the first step, we use a two-parameter Wiener process to regularise \eqref{eq:IntPathByPath} on dyadic intervals. In the second step we show a Gronwall type lemma (see Lemma \ref{lem:GronwallSheetd}). The regularisation is as follows: For any positive integer $n$, we divide $[0,1]$ into $2^{n}$ intervals $I_{n,k}=]k2^{-n},(k+1)2^{-n}]$ and define $\varrho_{nk\kp}$ by
\begin{align*}
	\varrho_{nk\kp}(x,y):=\int_{I_{n,\kp}}\int_{I_{n,k}}\{b(s,t,W_{s,t}+x)-b(s,t,W_{s,t}+y)\}\,\mathrm{d}t\mathrm{d}s.
\end{align*}
The next three lemmas whose proofs are given in Section \ref{Auxrel}  provide an estimate for $\varrho_{nk\kp}(x,y)$ and $\varrho_{nk\kp}(0,x)$ for every dyadic numbers $x,y\in\mathbf{Q}$. Lemmas \ref{lem:PseudoMetric1} and \ref{lem:PseudoMetric2} are counterparts of Lemmas 3.1 and 3.2 in \cite{Da07} for the Brownian sheet. The proof of Lemma \ref{lem:PseudoMetric1} uses  the local time-space integration formula for the Brownian sheet as given in \cite{BDM21}. Lemma \ref{lem:PseudoMetric3} follows from Lemma \ref{lem:PseudoMetric2} using the fact that the set of dyadic numbers is dense in $\R$.
\begin{lemma}\label{lem:PseudoMetric1}
	Suppose $b:\,[0,1]^2\times\R^d\to\R$ is a Borel measurable function such that $|b(s,t,x)|\leq1$ everywhere on $[0,1]^2\times\R^d$. Then there exists a subset $\Omega_{1}$ of $\Omega$ with $\Pb(\Omega_{1})=1$ such that for all $\omega\in\Omega_1$,
	\begin{align*}
		|\varrho_{nk\kp}(x,y)(\omega)|\leq C_1(\omega)2^{-n}\Big[\sqrt{n}+\Big(\log^+\frac{1}{|x-y|}\Big)^{1/2}\Big]|x-y|\,\text{ on }\Omega_{1}
	\end{align*}
	for all dyadic numbers $x,\,y\in\mathbf{Q}$ and all choices of integers $n,\,k,\,\kp$ with $n\geq1$, $0\leq k,\kp\leq 2^n-1$, where $C_1(\omega)$ is a positive random constant that does not depend on $x$, $y$, $n$, $k$ and $\kp$.
\end{lemma}

\begin{lemma}\label{lem:PseudoMetric2}
	Suppose $b$ is as in Lemma \ref{lem:PseudoMetric1}. Then there exists a subset $\Omega_{2}$ of $\Omega$ with $\Pb(\Omega_{2})=1$  such that for all $\omega\in\Omega_2$, for any choice of $n,\,k,\,\kp$, and any choice of a dyadic number $x\in\mathbf{Q}$
	\begin{align}\label{eq:coPseudoMetric2}
		\left|\varrho_{nk\kp}(0,x)(\omega)\right|\leq C_{2}(\omega)\sqrt{n}2^{-n}\Big(|x|+2^{-4^{n}}\Big),
	\end{align}
	where $C_2(\omega)$ is a positive random constant that does not depend on $x$, $n$, $k$ and $\kp$.
\end{lemma}
Observe that the proofs of the above two results do not require the monotonic argument on the drift  $b$.
\begin{lemma}\label{lem:PseudoMetric3}
	Suppose $b$ is as in Theorem \ref{maintheuniq2}. Let $\Omega_2$ be a subset of $\Omega$ such that, for any $\omega\in\Omega_2$, \eqref{eq:coPseudoMetric2} holds for every $n,\,k,\,\kp$, and every dyadic number $x\in\mathbf{Q}$. Then 
	\begin{align*} 
		\left|\varrho_{nk\kp}(0,x)(\omega)\right|\leq \widetilde{C}_{2}(\omega)\sqrt{n}2^{-n}\Big(|x|+2^{-4^{n}}\Big)
	\end{align*}
	for any $\omega\in\Omega_2$, any $n$, $k$, $\kp$, and any $x\in\mathbf{Q}$,  where $\widetilde{C}_2(\omega)$ is a positive random constant that does not depend on $x$, $n$, $k$ and $\kp$.
\end{lemma}

 The subsequent result is a Gronwall type lemma and constitutes the main result in the second step of the proof of Theorem \ref{theo:DavieSheetMonotoneD}. Its proof is found in Section 5.

\begin{lemma}\label{lem:GronwallSheetd}
	Let $W:=\Big(W^{(1)}_{s,t},\ldots,W^{(d)}_{s,t};(s,t)\in[0,1]^2\Big)$ be a $d$-dimensional Brownian sheet defined on a filtered probability space $(\Omega,\mathcal{F},\mathbb{F},\Pb)$, where $\mathbb{F}=(\mathcal{F}_{s,t};(s,t)\in[0,1]^2)$ and let
	 the drift $b$ be as in Theorem \ref{maintheuniq2}. There exists $\Omega_1\subset\Omega$ with $\Pb(\Omega_1)=1$ and a positive random constant $C_1$ such that for any $\omega\in\Omega_1$, any sufficiently large positive integer $n$, any $(k,\kp)\in\{1,2,\ldots,2^n\}^2$, any $\beta(n)\in\Big[2^{-4^{3n/4}},2^{-4^{2n/3}}\Big]$, and any solution   
	$u$ to the system of integral equations
	\begin{align}\label{eq:DavieSheetGronwalld}
		&u_i(s,t)-u_i(s,0)-u_i(0,t)+u_i(0,0)\nonumber\\
		&=\int_0^t\int_0^s\{b_i(\xi,\zeta,W_{\xi,\zeta}(\omega)+u(\xi,\zeta))-b_i(\xi,\zeta,W_{\xi,\zeta}(\omega))\}\mathrm{d}\xi\mathrm{d}\zeta, 
		\text{ }\forall\,(s,t)\in[0,1]^2,\,\forall\,i,
	\end{align}
satisfying
\begin{equation}\label{eq:GronwallInitialVd}
	\max\{|u(s,0)|,|u(0,t)|\}\leq\beta(n),\quad\forall\,(s,t)\in[0,1]^2,
\end{equation}
we have
\begin{align}\label{eq:GronwallSheetEstd}
	\max\{|\oun(k,\kp)|,|\uun(k,\kp)|\}\leq(3\sqrt{d})^{k+\kp-1}\Big(1+C_1(\omega)\sqrt{dn}2^{-n}\Big)^{k+\kp}\beta(n)\,\text{ on }\Omega_1,
\end{align}
	where $\oun(k,\kp)=(\oun_1(k,\kp),\ldots,\oun_d(k,\kp))$, $\uun(k,\kp)=(\uun_1(k,\kp),\ldots,\uun_{\,d}(k,\kp))$,
	\begin{align*}
		\oun_i(k,\kp):=\sup\limits_{(s,t)\in I_{n,k-1}\times I_{n,\kp-1}}\max\{0,u_i(s,t)\}=\sup\limits_{(s,t)\in I_{n,k-1}\times I_{n,\kp-1}}u_i^+(s,t),\quad\forall\,i
	\end{align*}
	and
	\begin{align*}
		\uun_i(k,\kp):=\sup\limits_{(s,t)\in I_{n,k-1}\times I_{n,\kp-1}}\max\{0,-u_i(s,t)\}=\sup\limits_{(s,t)\in I_{n,k-1}\times I_{n,\kp-1}}u_i^-(s,t),\quad\forall\,i.
	\end{align*}  

\end{lemma}
We are now ready to prove Theorem \ref{theo:DavieSheetMonotoneD}.

\begin{proof}[Proof of Theorem \ref{theo:DavieSheetMonotoneD}]
	We choose $\Omega_1$, $C_1$, $\omega$, $n$ and $\beta(n)$ as in Lemma \ref{lem:GronwallSheetd}. Let $u$ be a solution to \eqref{eq:IntPathByPath}. We have $\max\{|u|(s,0),|u|(0,t)\}=0\leq\beta(n)\text{ for all }(s,t)\in[0,1]^2.$ Moreover, we deduce from (\ref{eq:GronwallSheetEstd}) that
	\begin{align}\label{eq:GrowllUniqBoundd}
		\sup\limits_{k,\kp\in\{0,1,2,\cdots,2^n\}}\max\left\{|\oun(k,\kp)|,|\uun(k,\kp)|\right\}\leq (4\sqrt{d})^{2^{n+1}}\beta(n)
	\end{align}
	for all $n$ satisfying $C_1(\omega)\sqrt{dn}2^{-n}\leq1/3$. Since the right side of \eqref{eq:GrowllUniqBoundd} converges to $0$ as $n$ goes to $\infty$, it holds $u(s,t)=0$ on $\Omega_1$ for all $(s,t)$.
\end{proof}

\begin{proof}[Proof of Theorem \ref{maintheuniq1}]
	For every positive integer $n$, consider the bounded and nondecreasing function $g_n:\,\R\to\R$ defined by
	$$g_n(a)=
	\begin{cases}
		a \text{ for }|a|<n,\\
		n \text{ for } a\geq n,\\
		-n \text{ for } a\leq - n.
	\end{cases}
	$$
	Then for every $n$ and $i$, $g_n\circ b_i$ is a bounded and nondecreasing function.
	Let $b^{(n)}:\,[0,1]^2\times\R^d\to\R^d$ be the bounded and componentwise nondecreasing function given by $b^{(n)}_i(s,t,x)=g_n(b_i(s,t,x))$ for all $i$, all $(s,t)\in[0,1]^2$ and all $x\in\R^d$. We have 
	$$|b^{(n)}(s,t,x)|\leq M(1+|x|),\quad\forall\,n,\,(s,t)\in[0,1]^2,\,x\in\R^d.$$  
	It follows from Theorem \ref{maintheuniq2} that for any $n$, there exists
	an event $\Omega_n$ of full measure, such that for any $\omega\in\Omega_n$, the system of integral equations
	\begin{align}\label{eq:syst1}
		X_{s,t}(\omega)=\int_0^t\int_0^sb^{(n)}(\xi,\zeta,X_{\xi,\zeta}(\omega))\mathrm{d}\xi \mathrm{d}\zeta+W_{s,t}(\omega),\quad(s,t)\in[0,1],
	\end{align}
	has a unique solution $(X_{s,t}^{(n)}(\omega),0\leq s,t\leq1)$. Moreover, we have
	\begin{align}\label{eq:abssyst1}
		\Big|X_{s,t}^{(n)}(\omega)\Big|\leq M\int_0^t\int_0^s\Big|X_{\xi,\zeta}^{(n)}(\omega)\Big|\mathrm{d}\xi \mathrm{d}\zeta+M+|W_{s,t}(\omega)|.
	\end{align}	
	 By the Gronwall inequality for integrals in the plane provided in \cite[Section 2]{Sno72} (see also \cite[Section 1]{Ra76}), it holds
		\begin{align}\label{ineq:Snow1}
			\Big|X_{s,t}^{(n)}(\omega)\Big|\leq M+|W_{s,t}(\omega)|+M\int_0^t\int_0^s\Big(M+|W_{\xi,\zeta}(\omega)|\Big)h(\xi,\zeta,s,t)\mathrm{d}\xi \mathrm{d}\zeta
		\end{align}
		for all $(s,t)\in[0,1]^2$, where $h$ is the unique solution to
		\begin{align}\label{eq:Snow1D}
			h(\xi,\zeta,s,t)=1+M\int_{\xi}^t\int_{\zeta}^sh(\eta,\gamma,s,t)\mathrm{d}\eta \mathrm{d}\gamma,\quad(\xi,\gamma)\in[0,s]\times[0,t].
		\end{align}
		It is known (see for example \cite[Page 145]{Nu06}) that for every $(\xi,\zeta,s,t)$, $h(\xi,\zeta,s,t)=I_0\Big(2\sqrt{M(s-\xi)(t-\zeta)}\Big)$, where $I_0$ is the modified Bessel function of order zero. 
		Since $h$ is nonnegative, we have
		\begin{align*}
			\sup\limits_{(s,t)\in[0,1]^2}\Big|X^{(n)}_{s,t}(\omega)\Big|\leq C_1^{\ast}\Big(M+\sup\limits_{(s,t)\in[0,1]^2}|W_{s,t}(\omega)|\Big),
		\end{align*}
		with
		\begin{align*}
			C_1^{\ast}=1+MI_{0}(2\sqrt{M}).
		\end{align*}
	Let $\Omega_{\infty}=\bigcap\limits_{n\geq1}\Omega_n$, then $\Pb(\Omega_{\infty})=1$. Fix $\omega\in\Omega_{\infty}$ and $n\geq1$ such that
	\begin{align}\label{eq:SolutionBound}
		(C_1^{\ast})^2M\Big(1+\sup\limits_{(s,t)\in[0,1]^2}|W_{s,t}(\omega)|\Big)\leq n.
	\end{align}
	Since $C_1^{\ast}\geq1+M$, we obtain
	\begin{align*}
		\sup\limits_{(s,t)\in[0,1]^2}\Big|b(s,t,X^{(n)}_{s,t}(\omega))\Big|&\leq M\Big(1+\sup\limits_{(s,t)\in[0,1]^2}\Big|X^{(n)}_{s,t}(\omega)\Big|\Big)\\
		&\leq M\Big(1+C_1^{\ast}M+C_1^{\ast}\sup\limits_{(s,t)\in[0,1]^2}|W_{s,t}(\omega)|\Big)\leq n.
	\end{align*}
	As a consequence, $b^{(n)}(s,t,X^{(n)}_{s,t}(\omega))=b(s,t,X^{(n)}_{s,t}(\omega))$ for every $(s,t)\in[0,1]^2$. Hence $(X^{(n)}_{s,t}(\omega),0\leq s,t\leq1)$ is a solution to the system
	\begin{align}\label{eq:syst2}
		X_{s,t}(\omega)=\int_0^t\int_0^sb(\xi,\zeta,X_{\xi,\zeta}(\omega))\mathrm{d}\xi \mathrm{d}\zeta+W_{s,t}(\omega),\quad(s,t)\in[0,1].
	\end{align}
	Let $(y_{s,t},0\leq s,t\leq1)$ be another solution to \eqref{eq:syst2} for the same $\omega\in\Omega_{\infty}$. Then, using once more Gronwall inequality for integrals on the plane, we obtain
	\begin{align*}
		\sup\limits_{(s,t)\in[0,1]^2}|y_{s,t}|\leq C_1^{\ast}\Big(M+\sup\limits_{(s,t)\in[0,1]^2}|W_{s,t}(\omega)|\Big).
	\end{align*}
	This implies that for $n$ in \eqref{eq:SolutionBound}, $(y_{s,t},0\leq s,t\leq1)$ is also a solution to \eqref{eq:syst1}. Since $\omega\in\Omega_n$, the system \eqref{eq:syst1} has a unique solution for this $\omega$. Thus, $y_{s,t}=X^{(n)}_{s,t}(\omega)$ for every $(s,t)\in[0,1]^2$ and uniqueness is proved.  
\end{proof}

\section{Preliminary results}\label{prelresul}
In order to prove the auxiliary lemmas provided in the previous section, we need some preliminary results that have been obtained by applying a local time-space integration formula for Brownian sheets (see \cite{BDM21} for related results).
Let us first recall the notion of local time in the plane of the Brownian sheet.  Let $(W^{(0)}_{s,t},(s,t)\in D)$ be a one dimensional Brownian sheet given on a filtered probability space. For $s$ fixed, $(W^{(0)}_{s,t},t\in[0,1])$ is a one dimensional Brownian motion and its local time process $(L_1^x(s,t);x\in\R,t\geq0)$ is given by the Tanaka's formula (see for example \cite[Section 1]{Wa78}): 
\begin{align}\label{eq:TanakaSheet1}
	\int_0^t\mathbf{1}_{\{W^{(0)}_{s,u}\leq x\}}\mathrm{d}_uW^{(0)}_{s,u}=\frac{s}{2}L^x_1(s,t)-(W^{(0)}_{s,t}-x)^{-}+x^{+}.
\end{align}
For $s\in[0,1]$ fixed, let $\widehat{W}^{(0)}_{s,\cdot}$ be the time reversal process on $[0,1]$ of the Brownian motion $\widehat{W}^{(0)}_{s,\cdot}$ (i.e.,$\widehat{W}^{(0)}_{s,t}=W^{(0)}_{s,1-t}$) and let $(\widehat{L}^x_1(s,t);x\in\R,0\leq t\leq1)$ be the local time process  of $(\widehat{W}^{(0)}_{s,t},0\leq t\leq 1)$. Then the following holds
\begin{align*}
	\widehat{L}^x_1(s,t)=L^x_1(s,1)-L^x_1(s,1-t).
\end{align*}
Next, we consider the local time process in the plane $L:=(L_{s,t}^x;x\in\R,s\geq0,t\geq0)$ as defined in \cite[Section 2]{Wa78} (see also \cite[Section 6, Page 157]{CW75}) by
\begin{align*}
	L^x_{s,t}:=\int_0^sL_1^x(\xi,t)\mathrm{d}\xi,\quad\forall\,x\in\R,\,\forall\,(s,t)\in\R_+^2.
\end{align*} 
Then it holds
\begin{align}\label{eq:StoInt2DLocTime01}
	L^x_{s,t}=\int_0^s\int_{1-t}^1\mathbf{1}_{\{W^{(0)}_{\xi,u}\leq x\}}\frac{\mathrm{d}_uW^{(0)}_{\xi,u}}{\xi}\mathrm{d}\xi+\int_0^s\int_{0}^t\mathbf{1}_{\{\widehat{W}^{(0)}_{\xi,u}\leq x\}}\frac{\mathrm{d}_u\widehat{W}^{(0)}_{s,u}}{\xi}\mathrm{d}\xi,\quad\forall\,(s,t)\in[0,1]^2.
\end{align}
Let us now consider the norm $\Vert\cdot\Vert$ defined by
\begin{align*}
	\Vert f\Vert:=&2\Big(\int_0^1\int_0^1\int_{\R}f^2(s,t,x)\exp\Big(-\frac{x^2}{2st}\Big)\frac{\mathrm{d}x\mathrm{d}s\mathrm{d}t}{\sqrt{2\pi st}}\Big)^{1/2}\\
	&+\int_0^1\int_0^1\int_{\R}|xf(s,t,x)|\exp\Big(-\frac{x^2}{2st}\Big)\frac{\mathrm{d}x\mathrm{d}s\mathrm{d}t}{st\sqrt{2\pi st}}\\
	=&2\Big(\int_0^1\int_0^1\E\Big[f^2(s,t,W^{(0)}_{s,t})\Big]\mathrm{d}s\mathrm{d}t\Big)^{1/2}+\int_0^1\int_0^1\E\Big[\Big|f(s,t,W^{(0)}_{s,t})\frac{W^{(0)}_{s,t}}{st}\Big|\Big] \mathrm{d}s\mathrm{d}t.
\end{align*}
Consider the set $\mathcal{H}$ of measurable functions $f$ on $[0,1]^2\times\R$ such that $\Vert f\Vert<\infty$. Endowed with $\Vert\cdot\Vert$, the space  $\mathcal{H}$ is a Banach space. In the following, we define a stochastic integral over the space with respect to the local time for the elements of $\mathcal{H}$. This extends the definition  in \cite{Ei00}.
We say that $f_{\Delta}:\,[0,1]^2\times\R\to\R$ is an elementary function if there exist two sequences of real numbers $(x_i)_{0\leq i\leq n}$, $(f_{ijk};0\leq i\leq n,0\leq j\leq m, 0\leq k\leq\ell)$ and two subdivisions of $[0,1]$ $(s_j)_{0\leq j\leq m}$, $(t_k)_{0\leq k\leq \ell}$ such that
\begin{align}\label{eq:ElemFunction}
	f_{\Delta}(s,t,x)=\sum\limits_{(x_i,s_j,t_k)\in\Delta}f_{ijk}\mathbf{1}_{(x_i,x_{i+1}]}(x)\mathbf{1}_{(s_j,s_{j+1}]}(s)\mathbf{1}_{(t_k,t_{k+1}]}(t),
\end{align}
where $\Delta=\{(x_i,s_j,t_k);0\leq i\leq n,0\leq j\leq m, 0\leq k\leq\ell\}$.
\begin{defi}
	For a simple function $f_{\Delta}$ given in \eqref{eq:ElemFunction}, we define its integral with respect to $L$ as 
	\begin{align*}
		\int_0^1\int_0^1\int_{\R}f_{\Delta}(s,t,x)\mathrm{d}L^x_{s,t}:=&\sum\limits_{(x_i,s_j,t_k)\in\Delta}f_{ijk}\Big(L^{x_{i+1}}_{s_{j+1},t_{k+1}}-L^{x_{i+1}}_{s_{j},t_{k+1}}-L^{x_{i}}_{s_{j+1},t_{k+1}}+L^{x_{i}}_{s_{j},t_{k+1}}\\
		&\text{ }-L^{x_{i+1}}_{s_{j+1},t_{k}}+L^{x_{i+1}}_{s_{j},t_{k}}+L^{x_{i}}_{s_{j+1},t_{k}}-L^{x_{i}}_{s_{j},t_{k}}\Big).
	\end{align*}  
\end{defi}
\begin{remark}
	Let $f$ be an element of $\mathcal{H}$ and let $(f_n)_{n\in\N}$ be a sequence of elementary functions converging to $f$ in $\mathcal{H}$. It is proved in \cite[Proposition 2.1]{BDM21} that the sequence $\left(\int_0^1\int_0^1\int_{\R}f_n(s,t,x)\mathrm{d}L^x_{s,t}\right)_{n\in\N}$ converges in $L^1(\Omega,\Pb)$ and that the limit does not depend on the choice of the sequence $(f_n)_{n\in\N}$. This limit is called integral of $f$ with respect to $L$. Similar results were obtained in \cite{Ei00}.
\end{remark}

Let $f:\,[0,1]^2\times\R^d\to\R$ be a continuous function such that for any $(s,t)\in[0,1]^2$, $f(s,t,\cdot)$ is differentiable and for any $i\in\{1,\cdots,d\}$, the partial derivative $\partial_{x_i}f$ is continuous. We also know from \cite[Proposition 3.1]{BDM21} that for a $d$-dimensional Brownian sheet $\Big(W_{s,t}:=(W_{s,t}^{(1)},\cdots,W_{s,t}^{(d)});s\geq0,t\geq0\Big)$ defined on a filtered probability space and for any $(s,t)\in[0,1]^2$ and any $i\in\{1,\cdots,d\}$, we have
\begin{align}\label{eq:EisenSheetdD01}
	&\int_0^s\int_0^t\partial_{x_i}f(\xi,u,W_{\xi,u})\mathrm{d}u\mathrm{d}\xi\notag\\
	=&-\int_0^s\int_0^tf(\xi,u,W_{\xi,u})\frac{d_uW^{(i)}_{\xi,u}}{\xi}\mathrm{d}\xi-\int_0^s\int_{1-t}^1f(\xi,1-u,\widehat{W}_{\xi,u})\frac{d_uB^{(i)}_{\xi,u}}{\xi}\mathrm{d}\xi\nonumber\\
	&+\int_0^s\int_{1-t}^1f(\xi,1-u,\widehat{W}_{\xi,u})\frac{\widehat{W}^{(i)}_{\xi,u}}{\xi(1-u)}\mathrm{d}u\mathrm{d}\xi,
\end{align}
where $\widehat{W}^{(i)}:=(\widehat{W}^{(i)}_{\xi,u};0\leq\xi,u\leq1)$ and $B^{(i)}:=(B^{(i)}_{\xi,u};0\leq\xi,u\leq1)$ is a standard Brownian sheet with respect to the filtration of $\widehat{W}^{(i)}$, independent of $(W^{(i)}_{s,1},s\geq0)$. 

The following result will be extensively used in this work and corresponds to \cite[Proposition 2.1]{Sh16} for the standard Wiener process.

\begin{prop}\label{prop:DavieSheet1dd}
	Let $W:=\left(W^{(1)}_{s,t},\ldots,W^{(d)}_{s,t};(s,t)\in[0,1]^2\right)$ be a $\R^d$-valued Brownian sheet defined on a filtered probability space $(\Omega,\mathcal{F},\mathbb{F},\Pb)$, where $\mathbb{F}=(\mathcal{F}_{s,t};s,t\in[0,1])$. Let $b\in\mathcal{C}\left([0,1]^2,\mathcal{C}^1(\R^d)\right)$, $\Vert b\Vert_{\infty}\leq1$.
	Let $(a,\pa,\ve,\vep)\in[0,1]^4$. Then there exist positive constants $\alpha$ and $C$ (independent of $\nabla_yb$, $a$, $\pa$, $\ve$ and $\vep$) such that 
	\begin{align}\label{eq:DavieSheet0dd}
		\E\Big[\exp\Big(\alpha\vep \ve \Big|\int_0^1\int_0^1\nabla_yb\left(s,t,\widetilde{W}^{\ve,\vep}_{s,t}\right)\mathrm{d}t\mathrm{d}s\Big|^2\Big)\Big]\leq C.
	\end{align} 
	Here $\nabla_yb$ denotes the gradient of $b$ with respect to the third variable, $|\cdot|$ is the usual norm on $\R^d$ and  the $\R^d$-valued two-parameter Gaussian process $\widetilde{W}^{\ve,\vep}:=\Big(\widetilde{W}^{(\ve,\vep,1)}_{s,t},\ldots,\widetilde{W}^{(\ve,\vep,d)}_{s,t};(s,t)\in[0,1]^2\Big)$ is given by 
	$$
	\widetilde{W}^{(\ve,\vep,i)}_{s,t}=W^{(i)}_{\pa+\vep s,a+\ve t}-W^{(i)}_{\pa,a+\ve t}-W^{(i)}_{\pa+\vep s,a}+W^{(i)}_{\pa,a}\quad\text{for all }i\in\{1,\ldots,d\}.
	$$ 
\end{prop}
\begin{proof}
	The proof of \eqref{eq:DavieSheet0dd} is based on the local time-space integration formula \eqref{eq:EisenSheetdD01} and the Barlow-Yor inequality. Fix $(a,\pa,\ve,\vep)\in[0,1]^4$. Since $x\longmapsto e^{\alpha\ve\vep dx^2}$ is a convex function, we deduce from the Jensen inequality that 
	\begin{align*}
		&\E\Big[\exp\Big(\alpha\vep \ve \Big|\int_0^1\int_0^1\nabla_yb\Big(s,t,\widetilde{W}^{\vep,\ve}_{s,t}\Big)\mathrm{d}t\mathrm{d}s\Big|^2\Big)\Big]\\
		=&\E\Big[\exp\Big(\alpha\vep \ve \sum\limits_{i=1}^d\Big|\int_0^1\int_0^1\partial_{y_i}b\Big(s,t,\widetilde{W}^{\vep,\ve}_{s,t}\Big)\mathrm{d}t\mathrm{d}s\Big|^2\Big)\Big]\\
		\leq&\frac{1}{d}\sum\limits_{i=1}^d\E\Big[\exp\Big(\alpha d\vep \ve\Big|\int_0^1\int_0^1\partial_{y_i}b\Big(s,t,\widetilde{W}^{\vep,\ve}_{s,t}\Big)\mathrm{d}t\mathrm{d}s\Big|^2\Big)\Big].
	\end{align*}
	In order to obtain \eqref{eq:DavieSheet0dd} it suffices to prove that for every $i\in\{1,2,\ldots,d\}$, there exist positive constants $\alpha=\alpha_i$ and $C=C_i$  such that 
	\begin{align*}
		\E\Big[\exp\Big(\alpha \ve \vep\left|\int_0^1\int_0^1\partial_{y_i} b\Big(s,t,\widetilde{W}^{\vep,\ve}_{s,t}\Big)\mathrm{d}t\mathrm{d}s\Big|^2\Big)\right]\leq C.
	\end{align*}
	For every $i\in\{1,\ldots,d\}$, we apply \eqref{eq:EisenSheetdD01} to the standard $d$-dimensional Brownian sheet $\Big(Y_{s,t}:=(\ve\vep)^{-1/2}\widetilde{W}^{\vep,\ve}_{s,t},(s,t)\in[0,1]^2\Big)$ and the function $f:\,[0,1]^2\times\R^d\to\R$ given by $f(s,t,y)=b\Big(s,t,\sqrt{\ve\vep}\,y\Big)$ to obtain
	\begin{align*}
		&\int_0^1\int_0^1\partial_{y_i}f(s,t,Y_{s,t})\mathrm{d}t\mathrm{d}s\\
		=&
		-\int_0^1\int_0^1f(s,t,Y_{s,t})\frac{\mathrm{d}_tY^{(i)}_{s,t}}{s}\mathrm{d}s-\int_0^1\int_{0}^1f(s,1-t,Y_{s,1-t})\frac{\mathrm{d}_tB^{(i)}_{s,t}}{s}\mathrm{d}s\nonumber\\
		&+\int_0^1\int_{0}^1f(s,1-t,Y_{s,1-t})\frac{Y^{(i)}_{s,1-t}}{s(1-t)}\mathrm{d}t\mathrm{d}s,
	\end{align*} 
	where $(B^{(i)}_{s,t},0\leq s,t\leq1)$ denotes a standard Brownian sheet independent of the process $(Y^{(i)}_{s,1},0\leq s\leq 1)$. Hence,
	\begin{align*}
		&\sqrt{\ve \vep}\int_0^1\int_0^1\partial_{y_i}b(s,t,\widetilde{W}^{\vep,\ve}_{s,t})\mathrm{d}t\mathrm{d}s\\
		=&-\int_0^1\int_0^1b(s,t,\widetilde{W}^{\vep,\ve}_{s,t})\frac{\mathrm{d}_tY^{(i)}_{s,t}}{s}\mathrm{d}s-\int_0^1\int_{0}^1b(s,1-t,\widetilde{W}^{\vep,\ve}_{s,1-t})\frac{\mathrm{d}_tB^{(i)}_{s,t}}{s}\mathrm{d}s\\
		&+\int_0^1\int_{0}^1b(s,1-t,\widetilde{W}^{\vep,\ve}_{s,1-t})\frac{Y^{(i)}_{s,1-t}}{s(1-t)}\mathrm{d}t\mathrm{d}s\\
		=&I_1+I_2+I_3.
	\end{align*}
	Using once more the convexity of the function $x\longmapsto e^{3\alpha x^2}$ for any $\alpha>0$, we obtain
	\begin{align*}
		&\E\Big[\exp\Big(\alpha \ve \vep\Big|\int_0^1\int_0^1\nabla_yb\Big(s,t,\widetilde{W}^{\vep,\ve}_{s,t}\Big)\mathrm{d}t\mathrm{d}s\Big|^2\Big)\Big]\notag\\
		=&\E\Big[\exp\Big(\alpha (I_1+I_2+I_3)^2\Big)\Big]\\
		\leq&\frac{1}{3}\Big(\E\Big[\exp(3\alpha I_1^2)\Big]+\E\Big[\exp(3\alpha I_2^2)\Big]+\E\Big[\exp(3\alpha I_3^2)\Big]\Big).
	\end{align*}
	Hence to get the desired estimate, we need to prove that for every $k\in\{1,2,3\}$, there exist positive constants $\alpha_k$ and $C_k$ such that $\E\left[\exp(\alpha_k I_k^2)\right]\leq C_k.$\\
For every $s\in]0,1]$, $\left(s^{-1/2}Y^{}_{s,v},0\leq v\leq1\right)$ is a standard Brownian motion with respect to the filtration $\mathcal{F}_{1,\cdot}:=(\mathcal{F}_{1,t},t\in[0,1])$. Therefore the process
		$$
		\Big(M_{s,t}:=\int_{0}^tb(s,v,\widetilde{W}^{\vep,\ve}_{s,v})\mathrm{d}_v\left[\frac{Y^{}_{s,v}}{\sqrt{s}}\right],0\leq t\leq1\Big)
		$$
		is an It\^o integral  with respect to $\mathcal{F}_{1,\cdot}$ and thus a square-integrable $\mathcal{F}_{1,\cdot}$-martingale.
		In addition, for any constant $\alpha\in\R_+$, the following expansion formula holds
		\begin{align*}
			\E\Big[\exp\Big(\alpha I_1^2\Big)\Big]
			=\E\Big[\exp\Big(\alpha\Big|\int_0^1M_{s,1}^{}\frac{\mathrm{d}s}{\sqrt{s}}\Big|^2\Big)\Big]=\sum\limits_{m=0}^{\infty}\frac{\alpha^m\E\Big[\Big|\displaystyle\int_0^1M_{s,1}^{}\frac{\mathrm{d}s}{\sqrt{s}}\Big|^{2m}\Big]}{m!}. 
		\end{align*}
		Moreover, by the Jensen inequality and the Barlow-Yor inequality applied to the martingale $(M_{s,t},t\in[0,1])$ (see for example \cite[Proposition 4.2]{BY82} and 
		\cite[Appendix]{CK91}), 
		there exists a universal constant $c_1$ (not depending on $m$) such that,
		\begin{align*}
			\E\Big[\Big|\int_0^1M_{s,1}^{}\frac{\mathrm{d}s}{\sqrt{s}}\Big|^{2m}\Big]&\leq4^m\int_0^1\E\Big[|M_{s,1}|^{2m}\Big]\frac{\mathrm{d}s}{2\sqrt{s}}\leq 4^m\int_0^1\E\Big[\Big(\sup\limits_{0\leq t\leq 1}|M_{s,t}|\Big)^{2m}\Big]\frac{\mathrm{d}s}{2\sqrt{s}}\\&\leq c_1^{2m}(8m)^m\int_0^1\E\Big[\langle M^{}_{s,\cdot}\rangle_1^m\Big]\frac{\mathrm{d}s}{2\sqrt{s}}\\
			&\leq c_1^{2m}(8m)^m\int_0^1\E\Big[\Big(\int_{0}^{1}b^2(s,t,\widetilde{W}^{\vep,\ve}_{s,t})\mathrm{d}t\Big)^m\Big]\frac{\mathrm{d}s}{2\sqrt{s}}
			\leq c_1^{2m}(8m)^m,
		\end{align*} 
		since $\Vert b\Vert_{\infty}\leq 1$. Thus,
		\begin{align*}
			\E\Big[\exp\Big(\alpha I_1^2\Big)\Big]
			=\E\Big[\exp\Big(\alpha\Big|\int_0^1M_{s,1}^{}\frac{\mathrm{d}s}{\sqrt{s}}\Big|^2\Big)\Big]=
			\sum\limits_{m=0}^{\infty}\frac{\Big(8\alpha c_1^2\Big)^mm^m}{m!}. 
		\end{align*} 
		The above expression if finite for $\alpha$ such that $8\alpha c_1^2e<1$, i.e. $\alpha<1/8c_1^2e$ (by ratio test).
		Hence, there exists positive constants $\alpha_1$ and $C_1$ such that
		$$
		\E\Big[\exp\Big(\alpha_1 I_1^2\Big)\Big]\leq C_1.
		$$ 

	Similarly for
	\begin{align*}
		I_2=-\int_0^1\int_{0}^1b(s,1-t,\widetilde{W}^{\vep,\ve}_{s,1-t})\mathrm{d}_tB^{(i)}_{s,t}\frac{\mathrm{d}s}{s},
	\end{align*}
	there exists positive constants $\alpha_2$ and $C_2$ such that
	\begin{align*}
		\E\Big[\exp\Big(\alpha_2  I_2^2\Big)\Big]
		\leq\E\Big[\exp\Big(\alpha_2\Big|\int_0^1\int_{0}^1b(s,1-t,\widetilde{W}^{\vep,\ve}_{s,1-t})\mathrm{d}_tB^{(i)}_{s,t}\frac{\mathrm{d}s}{s}\Big|^2\Big)\Big]\leq C_2.
	\end{align*}
	It remains to estimate the term $I_3$. By the Jensen inequality, we have
	\begin{align}
		\E\Big[\exp\Big(\frac{I_3^2}{64}\Big)\Big]
		=&\E\Big[\exp\Big\{\frac{1}{4}\Big(\int_0^1\int_{0}^1\frac{b(s,1-t,\widetilde{W}^{\vep,\ve}_{s,1-t})Y^{(i)}_{s,1-t}}{\sqrt{s(1-t)}}\frac{\mathrm{d}t\mathrm{d}s}{4\sqrt{s(1-t)}}\Big)^2\Big\}\Big]
		\nonumber\\
		\leq&\int_0^1\int_{0}^{1}\E\Big[\exp\Big\{\frac{1}{4}b^2(s,1-t,\widetilde{W}^{\vep,\ve}_{s,1-t})\Big|\frac{Y^{(i)}_{s,1-t}}{\sqrt{s(1-t)}}\Big|^2\Big\}\Big]\frac{\mathrm{d}t\mathrm{d}s}{4\sqrt{s(1-t)}}\nonumber\\
		\leq&\int_0^1\int_{0}^{1}\E\Big[\exp\Big(\frac{1}{4}\Big|\frac{Y^{(i)}_{s,1-t}}{\sqrt{s(1-t)}}\Big|^2\Big)\Big]\frac{\mathrm{d}t\mathrm{d}s}{4\sqrt{s(1-t)}}. \label{eq:EstimShap2}
	\end{align}
	Note that for every $(s,t)\in]0,1]\times[0,1[$, $\dfrac{Y^{(i)}_{s,1-t}}{\sqrt{s(1-t)}}$ is a standard normal random variable. Therefore \eqref{eq:EstimShap2} yields 
	$$
	\E\Big[\exp\Big(\frac{I_3^2}{64}\Big)\Big]\leq C_3.
	$$
	The proof of \eqref{eq:DavieSheet0dd} is completed by taking $\alpha=\min(\frac{1}{64},\alpha_2,\alpha_3)$.
\end{proof}
For every $0\leq a< h\leq 1$, $0\leq \pa< \hp\leq 1$ and for $(x,y)\in\mathbb{R}^d$ define the function $\varrho$ by: 
$$
\varrho(x,y):=\int_{\pa}^{\hp}\int_a^h\Big\{b(\xi,\zeta,W_{\xi,\zeta}+x)-b(\xi,\zeta,W_{\xi,\zeta}+y)\Big\}\mathrm{d}\zeta\mathrm{d}\xi.
$$
As a consequence of Proposition \ref{prop:DavieSheet1dd}, we have:
\begin{corollary}\label{corol:DavieSheet1dds1}
	Let $b:[0,1]^2\times\R^d\to\R$ be a bounded and Borel measurable function such that $\Vert b\Vert_{\infty}\leq1$. Let $\alpha$, $C$ and $\widetilde{W}^{\ve,\vep}$ be defined as in Proposition  \ref{prop:DavieSheet1dd}. Then the following two bounds are valid: 
	\begin{enumerate}
		\item For every $(x,y)\in\R^{2d}$, $x\neq y$ and every $(\ve,\vep)\in[0,1]^2$, we have
		\begin{align}\label{eq:DavieSheet02dd}
			\E\Big[\exp\Big(\frac{\alpha\vep \ve }{|x-y|^2}\Big|\int_0^1\int_0^1\left\{b(s,t,\widetilde{W}^{\vep,\ve}_{s,t}+x)-b(s,t,\widetilde{W}^{\vep,\ve}_{s,t}+y)\right\}\mathrm{d}t\mathrm{d}s\Big|^2\Big)\Big]\leq C.
		\end{align}
		\item  
		For any $(x,y)\in\R^2$ and any $\eta>0$, we have
	\begin{align}\label{eq:EstDavieSigma1}
			\Pb\left(|\varrho(x,y)|\geq\eta\sqrt{(h-a)(\hp-\pa)}|x-y|\right)
			\leq Ce^{-\alpha\eta^2}.
		\end{align}
	\end{enumerate}
\end{corollary}
\begin{proof} 
	We start by showing \eqref{eq:DavieSheet02dd}. Note that it is enough to show this when $b$ is compactly supported and differentiable. Indeed, if $b$ is not differentiable, then, since the set of compactly supported and differentiable functions is dense in $L^{\infty}([0,1]^2\times\R^d)$, there exists a sequence $(b_n,n\in\N)$ of compactly supported and differentiable functions which converges a.e. to $b$ on $[0,1]^2\times\R^d$ and the desired result will follow from the Vitali's convergence theorem.\\
	Using the mean value theorem and the Cauchy-Schwartz inequality, we have
	\begin{align*}
		&\Big|\int_0^1\int_0^1\Big\{b(s,t,\widetilde{W}^{\vep,\ve}_{s,t}+x)-b(s,t,\widetilde{W}^{\vep,\ve}_{s,t}+y)\Big\}\mathrm{d}t\mathrm{d}s\Big|^2\\
		=&\Big|\int_0^1\int_0^1\int_0^1\nabla_yb_{}(s,t,\widetilde{W}^{\vep,\ve}_{s,t}+y+\xi(x-y))\cdot(x-y)\mathrm{d}\xi \mathrm{d}t\mathrm{d}s\Big|^2\\
		\leq&|x-y|^2\Big|\int_0^1\int_0^1\int_0^1\nabla_yb_{}(s,t,\widetilde{W}^{\vep,\ve}_{s,t}+y+\xi(x-y))\mathrm{d}\xi \mathrm{d}t\mathrm{d}s\Big|^2.
	\end{align*}
	Using the Minkowski inequality, the Jensen inequality and Proposition \ref{prop:DavieSheet1dd} applied to the function $(s,t,z)\longmapsto b(s,t,z+y+\xi(x-y))$, we obtain
	\begin{align}
		&\E\Big[\exp\Big(\frac{\alpha \ve \vep}{|x-y|^2}\Big|\int_0^1\int_0^1\Big\{b(s,t,\widetilde{W}^{\vep,\ve}_{s,t}+x)-b(s,t,\widetilde{W}^{\vep,\ve}_{s,t}+y)\Big\}\mathrm{d}t\mathrm{d}s\Big|^2\Big)\Big]\nonumber\\
		=&\E\Big[\exp\Big(\frac{\alpha \ve \vep}{|x-y|^2}\Big|\int_0^1\int_0^1\int_0^1\nabla_yb(s,t,\widetilde{W}^{\vep,\ve}_{s,t}+y+\xi(x-y))\cdot(x-y)\mathrm{d}\xi \mathrm{d}t\mathrm{d}s\Big|^2\Big)\Big]\nonumber\\
		\leq&\E\Big[\exp\Big(\frac{\alpha \ve \vep}{|x-y|^2}\int_0^1\Big|\int_0^1\int_0^1\nabla_yb(s,t,\widetilde{W}^{\vep,\ve}_{s,t}+y+\xi(x-y))\cdot(x-y) \mathrm{d}t\mathrm{d}s\Big|^2\mathrm{d}\xi\Big)\Big]\nonumber\\
		\leq& \int_0^1\E\Big[\exp\Big(\alpha \ve \vep\Big|\int_0^1\int_0^1\nabla_yb(s,t,\widetilde{W}^{\vep,\ve}_{s,t}+y+\xi(x-y))\mathrm{d}t\mathrm{d}s\Big|^2\Big)\Big]\mathrm{d}\xi\leq C.\label{eq:DavieSheetEstgs}
	\end{align}
	This ends the proof of \eqref{eq:DavieSheet02dd}.
	
As for the proof of  \eqref{eq:EstDavieSigma1}, let $(x,y)\in\R^{2d}$ such that $x\neq y$ and set $\ve=h-a$ and$\vep=\hp-\pa$. Define $\widehat{b}$ by $\widehat{b}(s,t,x):=b(\pa+\vep s,a+\ve t,x)$ and additionally define the processes 
	$\widetilde{W}^{\vep,\ve}:=\Big(\widetilde{W}_{s,t}^{\ve,\vep},(s,t)\in[0,1]^2\Big)$ and $Z^{\vep,\ve}:=\Big(Z_{s,t}^{\vep,\ve},(s,t)\in[0,1]^2\Big)$, respectively by
	\begin{align*}
		\widetilde{W}_{s,t}^{\vep,\ve}=W_{\pa+\vep s,a+\ve t}-W_{\pa,a+\ve t}-W_{\pa+\vep s,a}+W_{\pa,a},
	\end{align*}
	and
	\begin{align*}
		Z_{s,t}^{\vep,\ve}=W_{\pa,a+\ve t}+W_{\pa+\vep s,a}-W_{\pa,a}.
	\end{align*}
	Then $\widetilde{W}^{\vep,\ve}$ and $Z^{\vep,\ve}$ are independent processes. To see this, observe that  $Z^{\vep,\ve}_{s,t}$ is $\mathcal{F}_{1,a}\vee\mathcal{F}_{\pa,1}$-measurable for every $(s,t)\in[0,1]^2$ and $\widetilde{W}^{\vep,\ve}$ is independent of $\mathcal{F}_{1,a}\vee\mathcal{F}_{\pa,1}$.
	Using the change of variable $(\xi,\zeta):=(\pa+\vep s,a+\ve t)$, we obtain
	
	\begin{align*}
		\varrho(x,y)&=\int_{\pa}^{\hp}\int_a^h\Big\{b(\xi,\zeta,W_{\xi,\zeta}+x)-b(\xi,\zeta,W_{\xi,\zeta}+y)\Big\}\mathrm{d}\zeta \mathrm{d}\xi\\
		=&\ve\vep\int_0^1\int_0^1\Big\{\widehat{b}(s,t,W_{\pa+\vep s,a+\ve t}+x)-\widehat{b}(s,t,W_{\pa+\vep s,a+\ve t}+y)\Big\}\mathrm{d}t\mathrm{d}s\\
		=&\ve\vep\int_0^1\int_0^1\Big\{\widehat{b}(s,t,\widetilde{W}_{s,t}^{\vep,\ve}+Z_{s,t}^{\vep,\ve}+x)-\widehat{b}(s,t,\widetilde{W}_{s,t}^{\vep,\ve}+Z_{s,t}^{\ve,\vep}+y)\Big\}\mathrm{d}t\mathrm{d}s.
	\end{align*}
	Taking the expectation on both sides after some operations and using the fact that $\widetilde{W}^{\vep,\ve}$ and $Z^{\vep,\ve}$ are independent, we have
	\begin{align*}
		&\E\Big[\exp\Big(\frac{\alpha|\varrho(x,y)|^2}{\ve\vep|x-y|^2}\Big)\Big]\\
		=&\E\Big[\exp\Big(\frac{\alpha\ve\vep}{|x-y|^2}\Big|\int_0^1\int_0^1\left\{\widehat{b}(s,t,\widetilde{W}_{s,t}^{\vep,\ve}+Z_{s,t}^{\ve,\vep}+x)-\widehat{b}(s,t,\widetilde{W}_{s,t}^{\vep,\ve}+Z_{s,t}^{\ve,\vep}+y)\right\}\mathrm{d}t\mathrm{d}s\Big|^2\Big)\Big]\\
		=&\int \E\Big[\exp\Big(\frac{\alpha\ve\vep}{|x-y|^2}\Big|\int_0^1\int_0^1\Big\{\widehat{b}^z_{s,t}(\widetilde{W}_{s,t}^{\vep,\ve}+x)-\widehat{b}^z_{s,t}(\widetilde{W}_{s,t}^{\vep,\ve}+y)\Big\}\mathrm{d}t \mathrm{d}s\Big|^2\Big)\Big]\Pb_{Z^{\ve,\vep}}(\mathrm{d}z),
	\end{align*}
	where $\widehat{b}^z_{s,t}(w)=\widehat{b}(s,t,w+z_{s,t})$. Since $|\widehat{b}(s,t,w)|\leq1$ almost everywhere on $[0,1]^2\times\R^d$, we deduce from \eqref{eq:DavieSheet02dd} that 
	\begin{align*}
		&\E\Big[\exp\Big(\frac{\alpha|\varrho(x,y)|^2}{\ve\vep|x-y|^2}\Big)\Big]\\
		=&\int \E\Big[\exp\Big(\frac{\alpha\ve\vep}{|x-y|^2}\Big|\int_0^1\int_0^1\Big\{\widehat{b}^z_{s,t}(\widetilde{W}_{s,t}^{\vep,\ve}+x)-\widehat{b}^z_{s,t}(\widetilde{W}_{s,t}^{\vep,\ve}+y)\Big\}\mathrm{d}t\mathrm{d}s\Big|^2\Big)\Big]\Pb_{Z^{\ve,\vep}}(\mathrm{d}z)\leq C.
	\end{align*}
	Therefore, by Chebychev inequality, we obtain
	\begin{align*}
		\Pb\Big(|\varrho(x,y)|\geq\eta\sqrt{\ve\vep}|x-y|\Big)\leq e^{-\alpha\eta^2}\E\Big[\exp\Big(\frac{\alpha|\varrho(x,y)|^2}{\ve\vep|x-y|^2}\Big)\Big]\leq Ce^{-\alpha\eta^2}
	\end{align*}
	for any $\eta>0$. This concludes the proof. 
\end{proof}

\section{Proof of the auxiliary results}\label{Auxrel}
In this section we prove the auxiliary results stated in Section \ref{sectmasinres}. Consider $\mathbf{Q}=[-1,1]^d$ and its dyadic decomposition.  

\subsection{Regularization by noise}\label{regbynoise}

\begin{proof}[Proof of Lemma \ref{lem:PseudoMetric1}]
	Let $\mathcal{Q}$ denote the set of couples $(x,y)$ of dyadic numbers in $\mathbf{Q}$. For every integer $m\geq0$, we define
	$$\mathcal{Q}_{m}=\left\{(x,y)\in\mathcal{Q}:\,2^mx\in\Z^d,\,2^my\in\Z^d \right\}.$$ Then   $\mathcal{Q}_{m}$ has no more than $2^{2d(m+2)}$ elements and we have $\mathcal{Q}=\bigcup\limits_{m\in\N}\mathcal{Q}_{m}$. Consider the set $\mathbf{E}_{\delta,n}$ defined by
	\begin{align*}
		\mathbf{E}_{\delta,n}:=&\left\{\omega\in\Omega:\,\text{there exist }k,\kp\in\{0,1,\ldots,2^n-1\},\text{ }m\in\N^{\ast},\text{ } \text{ and }(x,y)\in\mathcal{Q}_{m}\right.\\
		&\qquad \text{ such that }
		|\varrho_{nk\kp}(x,y)|(\omega)\geq\delta(1+\sqrt{n}+\sqrt{m})2^{-n}|x-y|\}
	\end{align*} 
	for every $n\in\N$, $\delta\in\Q_+$. Observe that
	\begin{align*}
		&\mathbf{E}_{\delta,n}\\
		=&\bigcup_{k=0}^{2^n-1}\bigcup_{\kp=0}^{2^n-1}\bigcup_{m=0}^{\infty}\Big(\bigcup_{(x,y)\in\mathcal{Q}_{m}}\left\{\omega\in\Omega:\,|\varrho_{nk\kp}(x,y)|(\omega)\geq\delta(1+\sqrt{n}+\sqrt{m})2^{-n}|x-y|\right\}\Big).
	\end{align*} 
	Define also $\mathbf{E}_{\delta}$ by $\mathbf{E}_{\delta}:=\bigcup\limits_{n=0}^{\infty}\mathbf{E}_{\delta,n}$.
	Then, we deduce from \eqref{eq:EstDavieSigma1} that 
	\begin{align*}
		\Pb(\mathbf{E}_{\delta})&\leq\sum\limits_{n=0}^{\infty}\sum\limits_{k=0}^{2^n-1}\sum\limits_{\kp=0}^{2^n-1}\sum\limits_{m=0}^{\infty} \sum\limits_{(x,y)\in\mathcal{Q}_{m}}\Pb\left(|\varrho_{nk\kp}(x,y)|\geq \delta(1+\sqrt{n}+\sqrt{m})2^{-n}|x-y|\right)\\
		&\leq 2^{4d}C \sum\limits_{n=0}^{\infty}\sum\limits_{m=0}^{\infty}2^{2n}2^{2dm}e^{-\alpha\delta^2(1+n+m)},
	\end{align*}
	where $C$ and $\alpha$ are the deterministic constants in Proposition \ref{prop:DavieSheet1dd}.
	Moreover, $(\mathbf{E}_{\delta},\delta\in\Q_+)$ is a nonincreasing family and for $\delta\geq\delta_0:=\sqrt{2(d+1)\alpha^{-1}}$, we have
	\begin{align}\label{eq:EstlimPEdelta00e}
		\Pb(\mathbf{E}_{\delta})\leq 2^{4d}Ce^{-\alpha\delta^2}\,\sum\limits_{n=0}^{\infty}\sum\limits_{m=0}^{\infty}2^{-dn-2m}\leq2^{4d}C\,e^{-\alpha\delta^2}.
	\end{align}
	Therefore
	\begin{align*}
		\Pb\Big(\bigcap\limits_{\delta\in[\delta_0,\infty[\cap\Q}\mathbf{E}_{\delta}\Big)=\lim\limits_{\delta\to\infty}\Pb(\mathbf{E}_{\delta})=0.
	\end{align*} 
	Let us now define $\Omega_{1}$ by
	$$
	\Omega_{1}:=\bigcup\limits_{\delta\in[\delta_0,\infty[\cap\Q}\left(\Omega\setminus\mathbf{E}_{\delta}\right).
	$$
	Then $\Pb(\Omega_{1})=1$ and for every $\omega\in\Omega_1$, there exists $\delta_{\omega}>0$ such that 
	\begin{align}\label{eq:DaviePseudoMetricc00}
		|\varrho_{nk\kp}(x,y)(\omega)|<\delta_{\omega} 2^{-n}(1+\sqrt{n}+\sqrt{m})|x-y|
	\end{align}
	for all choices of $n$, $k$, $\kp$, $m$, and all choices of couples $(x,y)$ in $\mathcal{Q}_{m}$.\\
	Now, fix $\omega\in\Omega_1$, choose two dyadic numbers $x$, $y$ in $\mathbf{Q}$ and define $m$ by
	\begin{align}
			m:=\inf\Big\{n\in \mathbb{N} \text{ s.t. }2^{-n}\leq \max_{1\leq i\leq d}|x_i-y_i|\Big\}.
	\end{align}

		For $r\geq m$, define $x_r=(2^{-r}\floor*{2^rx_1},\ldots,2^{-r}\floor*{2^rx_d})$ and $y_r=(2^{-r}\floor*{2^ry_1},\ldots,2^{-r}\floor*{2^ry_d})$, where $ \floor*{\cdot}$ is the integer part function. Observe that $(x_m,y_m)\in\mathcal{Q}_m$, $(x_r,x_{r+1})\in\mathcal{Q}_{r+1}$ and $(y_r,y_{r+1})\in\mathcal{Q}_{r+1}$. Since for any $i\in\{1,\ldots,d\}$, $|\floor*{2^mx_i}-\floor*{2^my_i}|\leq 1+2^m|x_i-y_i|$ and $(\floor*{2x_i}-2\floor*{x_i},\floor*{2y_i}-2\floor*{y_i})\in\{0,1\}^2$, we have $|x_m-y_m|\leq3\sqrt{d}\,2^{-m}$, $|x_{r+1}-x_r|\leq\sqrt{d}\,2^{-r-1}$ and $|y_{r+1}-y_r|\leq\sqrt{d}\,2^{-r-1}$.


 It follows from the definition of $\rho_{nk}$ that 
	\begin{align}\label{eq:AddPropRho1}
		\varrho_{nk\kp}(x,y)=\varrho_{nk\kp}(x,x_m)+\varrho_{nk\kp}(x_m,y_m)+\varrho_{nk\kp}(y_m,y).
	\end{align}
	Moreover since $\varrho_{nk\kp}(x_m,x_m)=0$ and $\varrho_{nk\kp}(y_m,y_m)=0$, we obtain
	\begin{align*}
		\varrho_{nk\kp}(x_{q+1},x_m)=\sum\limits_{r=m}^{q}\rho_{nk\kp}(x_{r+1},x_{r})\,\text{ and }\varrho_{nk\kp}(y_m,y_{q+1})=\sum\limits_{r=m}^{q}\varrho_{nk\kp}(y_{r},y_{r+1})
	\end{align*}
	for every integer $q\geq m+1$. In addition for some integer $q\geq m+1$, we have $x_r=x$ and $y_r=y$ for all $r\geq q$, therefore 
	\begin{align*}
		\varrho_{nk\kp}(x,x_m)=\sum\limits_{r=m}^{\infty}\varrho_{nk\kp}(x_r,x_{r+1})\,\text{ and }\varrho_{nk\kp}(y_m,y)=\sum\limits_{r=m}^{\infty}\varrho_{nk\kp}(y_{r+1},y_{r}).
	\end{align*}
	It follows from (\ref{eq:AddPropRho1}) and (\ref{eq:DaviePseudoMetricc00}) that
	\begin{align*}
		&2^{n}|\varrho_{nk\kp}(x,y)(\omega)|\\
		&\leq3\sqrt{d}\delta_{\omega}(1+\sqrt{n}+\sqrt{m})2^{-m}+2\sqrt{d}\delta_{\omega}\sum\limits_{r=m}^{\infty}(1+\sqrt{n}+\sqrt{r+1})2^{-r-1}\\
		&\leq3\sqrt{d}\delta_{\omega}(1+\sqrt{n}+\sqrt{m})2^{-m}+2\sqrt{d}\delta_{\omega}\sqrt{n}\sum\limits_{r=m}^{\infty}2^{-r-1}+\frac{4\sqrt{d}\delta_{\omega}}{\sqrt{m+1}}\sum\limits_{r=m}^{\infty}(r+1)2^{-r-1}.
	\end{align*} 
	Using the following facts:  $\sum\limits_{r=m}^{\infty}2^{-r-1}=2^{-m}$, $\sum\limits_{r=m}^{\infty}r2^{-r+1}=(m+1)2^{-m+2}$, $2^{-m}\leq\max\limits_{1\leq i\leq d}|x_i-y_i|\leq2^{-m+1}$, $|x-y|\leq\sqrt{d}\max\limits_{1\leq i\leq d}|x_i-y_i|\leq\sqrt{d}\,|x-y|$ and $n\geq1$, we obtain
	\begin{align*}
		&2^{n}|\varrho_{nk\kp}(x,y)(\omega)|\nonumber\\
		\leq& 24\sqrt{d}\delta_{\omega}(1+\sqrt{n}+\sqrt{m})2^{-m}\leq 48d\delta_{\omega} \Big[\sqrt{n}+\Big(\log^+\frac{1}{|x-y|}\Big)^{1/2}\Big]|x-y|.\nonumber\\
	\end{align*}
	Choosing  $C_1(\omega)=48d\delta_{\omega}$ yields the desired result. 
\end{proof}


\begin{proof}[Proof of Lemma \ref{lem:PseudoMetric2}]
It suffices to show that there exists a subset $\widetilde{\Omega}_1$ of $\Omega$ with $\Pb(\widetilde{\Omega}_1)=1$ and a positive random constant $\widetilde{C}_1$ such that for any $\omega\in\widetilde{\Omega}_1$, any $n\in\N^{\ast}$, any $k,\,\kp\in\{0,1,\cdots,2^n-1\}$ and any dyadic number $x\in\mathbf{Q}$ with $\max\limits_{1\leq i\leq d}|x_i|\geq2^{-2^{2n+1}}$,
		\begin{align}\label{EqPseudoMetric2}
			|\varrho_{nk\kp}(x,0)(\omega)|\leq \widetilde{C}_1 (\omega)\sqrt{n}2^{-n}\Big(|x|+2^{-4^{n}}\Big).
		\end{align}
		In fact, using Lemma \ref{lem:PseudoMetric1}, there exist a subset $\Omega_1$ of $\Omega$ with $\Pb(\Omega_1)=1$ and a positive random constant $C_1$ such that for any $\omega\in\Omega_{1}$, any $n$, $k$, $\kp$ and any dyadic number $x\in\mathbf{Q}$ with $\max\limits_{1\leq i\leq d}|x_i|<2^{-2^{2n+1}}$, we have
		\begin{align*}
			|\varrho_{nk\kp}(x,0)(\omega)|\leq & C_1(\omega)2^{-n}\Big(\sqrt{n}+\Big(\log^+\frac{1}{|x|}\Big)^{1/2}\Big)|x|\\
			\leq &C_1(\omega)2^{-n}\sqrt{n}\,|x|+ \sqrt[4]{d}C_1(\omega)2^{-n}\,2^{-4^{n}}\Big(|x|\log^+\frac{1}{|x|}\Big)^{1/2}
			\\
			\leq& C_1(\omega)\Big(1+\sqrt[4]{dC_0^2}\Big)\sqrt{n}2^{-n}\Big(|x|+2^{-4^{n}}\Big)
			\leq C_{1,1}(\omega)\sqrt{n}2^{-n}\Big(|x|+2^{-4^{n}}\Big),
		\end{align*}
		where $C_0=\sup_{\xi\in]0,1]}\xi\log^+(1/\xi)$ and $C_{1,1}(\omega)=2C_{1}(\omega)\left(1+\sqrt[4]{dC_0^2}\right)$. 
		The required result follows by choosing $\Omega_2=\Omega_1\cap\widetilde{\Omega}_1$ and $C_2=\max\{C_1,\widetilde{C}_1\}$.\\		
		Let us now prove that \eqref{EqPseudoMetric2} holds. Define the following three sets:
		\begin{align*}
			\mathcal{O}_q:=&\{y\in\mathbf{Q}:\,\max\limits_{1\leq i\leq d}|y_i|\leq 2^{-q}\},\notag\\
			\mathcal{D}_q:=&\left\{(y,z)\in\mathcal{O}_q^2 \text{ such that there exists } r\in\N,\,r\geq q\text{ s.t. }2^ry\in\Z^d,  2^rz\in\Z^d\right\},
		\end{align*}
		and for $r\geq q$, 
		$$
		\mathcal{D}_{q,r}=\left\{(y,z)\in\mathcal{O}_q^2:\,2^ry\in\Z^d,\,2^rz\in\Z^d\right\}.
		$$
	
	Observe that $\mathcal{D}_{q,r}$ has no more than $2^{4d}\times2^{2d(r-q)}$ elements and
	$$\mathcal{D}_q=\bigcup\limits_{r=q}^{\infty}\mathcal{D}_{q,r}.$$
	In addition, define 
	\begin{align*}
		\mathbf{F}_{\delta,n}:=&\left\{\omega\in\Omega:\,\text{for some }k\in\{0,1,\cdots,2^n-1\},\text{ }q\in\{0,1,\cdots,2^{2n+1}\},\text{ }r\geq q,\text{ and }\right.\\ &\quad\left.(y,z)\in\mathcal{D}_{q,r}\text{ one has }
		\left|\varrho_{nk\kp}(y,z)\right|(\omega)\geq\delta2^{-n}(\sqrt{n}+\sqrt{r-q})|y-z|\right\}
	\end{align*}
	for every $n\in\N^{\ast}$ and $\delta>0$. Then 
	\begin{align*}
		&\mathbf{F}_{\delta,n}=\\
		&\bigcup_{k=0}^{2^n-1}\bigcup_{\kp=0}^{2^n-1}\bigcup\limits_{q=0}^{2^{2n+1}}\bigcup_{r=q}^{\infty}\Big(\bigcup_{(y,z)\in\mathcal{D}_{q,r}}\left\{\omega\in\Omega:\,\left|\varrho_{nk\kp}(y,z)\right|(\omega)\geq\delta2^{-n}(\sqrt{n}+\sqrt{r-q})|y-z|\right\}\Big).
	\end{align*}
	
	Let $\mathbf{F}_{\delta}$ be defined by $\mathbf{F}_{\delta}:=\bigcup\limits_{n=1}^{\infty}\mathbf{F}_{\delta,n}$. Next we show that $\Pb(\mathbf{F}_{\delta})$ tends to $0$ as $\delta$ goes to $\infty$. By the definition of $\mathbf{F}_{\delta}$ and using \eqref{eq:EstDavieSigma1}, we have
	\begin{align*}
		\Pb(\mathbf{F}_{\delta})
		\leq&\sum\limits_{n=1}^{\infty}\sum\limits_{k=0}^{2^n-1} \sum\limits_{\kp=0}^{2^n-1}\sum\limits_{q=0}^{2^{2n+1}}\sum\limits_{r=q}^{\infty}\sum\limits_{(y,z)\in\mathcal{D}_{q,r}}\Pb\left(\left|\varrho_{nk\kp}(y,z)\right|\geq
		\delta 2^{-n}(\sqrt{n}+\sqrt{r-q})|y-z|\right)\\
		\leq& 2^{4d}C\sum\limits_{n=1}^{\infty}\sum\limits_{q=0}^{2^{2n+1}}\sum\limits_{r=q}^{\infty}2^{2n}2^{2d(r-q)}e^{-\alpha\delta^2(n+r-q)}\\
		\leq& 2^{4d+2}C \sum\limits_{n=1}^{\infty}\sum\limits_{r=0}^{\infty}2^{4n+2dr}e^{-\alpha\delta^2(n+r)}\\
		=&4^{2d+1}C \sum\limits_{n=0}^{\infty}\sum\limits_{r=0}^{\infty}2^{4n+2dr}e^{-\alpha\delta^2(1+n+r)},
	\end{align*}
	where $C$ and $\alpha$ are the deterministic constants in Proposition \ref{prop:DavieSheet1dd}. For $\delta\geq\delta^{\ast}:=\sqrt{2(d+2)\alpha^{-1}}$, 
	\begin{align*}
		\Pb(\mathbf{F}_{\delta})&\leq4^{2d+1}Ce^{-\alpha\delta^2} \sum\limits_{n=0}^{\infty}\sum\limits_{r=0}^{\infty}2^{4n+r}2^{-\alpha\delta^2(n+r)}\\
		&\leq4^{2d+1}Ce^{-\alpha\delta^2} \sum\limits_{n=0}^{\infty}\sum\limits_{r=0}^{\infty}2^{-2dn-4r}\leq4^{2d+1}Ce^{-\alpha\delta^2}.
	\end{align*}
	It follows that $\Pb(\mathbf{F}_{\delta})$ tends to $0$ as $\delta$ goes to $\infty$. Define $\widetilde{\Omega}_14=$ by
	$$
	\widetilde{\Omega}_1:=\bigcup\limits_{\delta\in[\delta^{\ast},\infty[\cap\Q}(\Omega\setminus\mathbf{F}_{\delta}).
	$$ 
	Observe that $(\mathbf{F}_{\delta},\delta\in[\delta^{\ast},\infty[\cap\Q)$ is a nonincreasing family. Then 
	\begin{align*}
		\Pb\left(\widetilde{\Omega}_1\right)=1-\lim\limits_{\delta\to\infty}\Pb(F_{\delta})=1
	\end{align*}
	and for every $\omega\in\widetilde{\Omega}_1$, there exists $\odelta_{\omega}>0$ such that
	\begin{align}\label{eq:PseudoMetric2}
		\left|\varrho_{nk\kp}(y,z)(\omega)\right|<\odelta_{\omega}(\sqrt{n}+\sqrt{r-q})2^{-n}|y-z|
	\end{align}
	for all choices of $n$, $k$, $\kp$, $q\in\{0,1,\ldots,2^{n+1}\}$, $r\geq q$ and couples $(y,z)$ in $\mathcal{D}_{q,r}$.
	
	Now, let $m$ be the largest nonnegative integer $m$ such that $x\in\mathcal{O}_m$, i.e.$2^{-m-1}<\max\limits_{1\leq i\leq d}|x_i|\leq2^{-m}$. For $r\geq m$ define $x_r=(x_{1,r},\ldots,x_{d,r})$, where $x_{i,r}=2^{-r}\floor*{2^rx_i}$. 
	   Then $(0,x_m)\in\mathcal{D}_{m,m}$ and for every $r\geq m$, $(x_r,x_{r+1})\in\mathcal{D}_{m,r+1}$. Moreover, $|x_m-0|\leq \sqrt{d}2^{-m}$ and for every $r\geq m$, $|x_r-x_{r+1}|\leq\sqrt{d}\,2^{-r-1}$. We deduce from \eqref{eq:AddPropRho1}, \eqref{eq:PseudoMetric2} and the inequality $\sum\limits_{r=0}^{\infty}\sqrt{r}2^{-r}\leq 4$ that
	\begin{align}
		2^{n}\Big|\varrho_{nk\kp}(x,0)(\omega)\Big|&\leq 2^{n} \Big(\Big|\varrho_{nk\kp}(x_m,0)(\omega)\Big|+
		\sum\limits_{r=m}^{\infty}\Big|\varrho_{nk\kp}(x_{r+1},x_r)(\omega)\Big|\Big)\nonumber\\
		\leq&\odelta_{\omega}\sqrt{dn}2^{-m}+\sqrt{d}\odelta_{\omega}\sum\limits_{r=m}^{\infty}\left(\sqrt{n}+\sqrt{r-m}\right)2^{-r-1}\nonumber\\
		\leq&6\odelta_{\omega}\sqrt{dn} 2^{-m}\leq 12\odelta_{\omega}\sqrt{dn}|x|\leq 12\odelta_{\omega}\sqrt{dn}\Big(|x|+2^{-4^{n}}\Big).\nonumber
	\end{align}
The proof of \eqref{EqPseudoMetric2} is completed by choosing $\widetilde{C}_1=12\sqrt{d}\,\overline{\delta}_{\omega}$.
\end{proof}

\begin{proof}[Proof of Lemma \ref{lem:PseudoMetric3}]
	 
Define $\varrho_{nk\kp}^+:=\max\{0,\varrho_{nk\kp}\}$, $\varrho_{nk\kp}^-:=\max\{0,-\varrho_{nk\kp}\}$, $x^+:=(x^+_1,\ldots,x^+_d)$ and $x^-:=(x^-_1,\ldots,x^-_d)$, where $x_i^+=\max\{0,x_i\}$ and $x_i^-=\max\{0,-x_i\}$ for every $i\in\{1,\cdots,d\}$. Consider the sequence $(x^+_r)_{r\in\N}=(x^+_{r,1},\ldots, x^+_{r,d})_{r\in\N}$ in $\mathbf{Q}$ defined by $x^+_{r,i}:=1-\floor*{(1-x^+_i)2^r}2^{-r},\,\,i=1,\ldots,d$. Observe that $(x^+_r,r\in\N)$ is a coordinatewise nonincreasing sequence of dyadic numbers in $\mathbf{Q}$ converging to $x^+$. It follows from the hypothesis on $b$ and \eqref{eq:coPseudoMetric2} that 
\begin{align*}
	\varrho^+_{nk\kp}(x,0)(\omega)\leq\varrho_{nk\kp}(x^+_r,0)(\omega)\leq C_2(\omega)\sqrt{n}2^{-n}\Big(|x^+_r|+2^{-4^n}\Big)\text{ for all }n,\,k,\,\kp,\,r.
\end{align*}
Similarly for $(x^-_r)_{r\in\N}=(x^-_{r,1},\ldots, x^-_{r,d})_{r\in\N}$ defined by
$x^-_{r,i}:=1-\floor*{(1-x^-_i)2^r}2^{-r},\,\,i=1,\ldots,d$, we also have
\begin{align*}
	\varrho^-_{nk\kp}(x,0)(\omega)\leq-\varrho_{nk\kp}(-x^-,0)(\omega)\leq-\varrho_{nk\kp}(-x^-_r,0)(\omega)\leq C_2(\omega)\sqrt{n}2^{-n}\Big(|x^-_r|+2^{-4^n}\Big)\text{ for all }n,\,k,\,\kp,\,r.
\end{align*}
Hence,
\begin{align*}
	|\varrho_{nk\kp}(x,0)(\omega)|=\varrho^+_{nk\kp}(x,0)(\omega)+\varrho^-_{nk\kp}(x,0)(\omega)\leq C_2(\omega)\sqrt{n}2^{-n}\Big(|x^+_r|+|x^-_r|+2\times2^{-4^n}\Big)\text{ for all }n,\,k,\,\kp,\,r.
\end{align*}
Thus, by taking the limit as $r$ goes to $\infty$, we get
\begin{align*}
	|\varrho_{nk\kp}(x,0)(\omega)|\leq 2C_2(\omega)\sqrt{n}2^{-n}\Big(|x|+2^{-4^n}\Big)\text{ for all }n,\,k,\,\kp.
\end{align*}
Taking $\widetilde{C}_2(\omega)=2C_2(\omega)$ gives the desired result.

\end{proof}

\subsection{A Gronwall type inequality} The second and last step to prove the uniqueness result is the  version of the Gronwall inequality given in Lemma \ref{lem:GronwallSheetd}. 
Recall that 
$$u_i^{+}=\max\{0,u_i\},\,\,u_i^{-}=\max\{0,-u_i\}$$ and set $u^{+}=\left(u_1^{+},\ldots,u_d^{+}\right)$ and $u^{-}=\left(u_1^{-},\ldots,u_d^{-}\right)$.
\begin{proof}[Proof of Lemma \ref{lem:GronwallSheetd}]
	Using Lemmas \ref{lem:PseudoMetric2} and \ref{lem:PseudoMetric3}, there exists a subset $\Omega_{2}$ of $\Omega$ with $\Pb(\Omega_{2})=1$  such that for all $\omega\in\Omega_2$, we have
	\begin{align}\label{eq:coPseudoMetricc3d}
		\left|\varrho_{nk\kp}^{(i)}(0,x)(\omega)\right|:=&\Big|\int_{I_{n,k}}\int_{I_{n,\kp}}\{b_i(\xi,\zeta,W_{\xi,\zeta}(\omega)+x)-b_i(\xi,\zeta,W_{\xi,\zeta}(\omega))\}\mathrm{d}\xi\mathrm{d}\zeta\Big|\nonumber\\&\leq \widetilde{C}_{2,i}(\omega)\sqrt{n}2^{-n}\left(|x|+\beta(n)\right)\leq\widetilde{C}_{2}(\omega)\sqrt{n}2^{-n}\left(|x|+\beta(n)\right)\,\text{ on }\Omega_{2}
	\end{align}
	for all $i\in\{1,\ldots,d\}$, all integers $n,\,k,\,\kp$ with $n\geq1$, $0\leq k,\kp\leq 2^n-1$ and all $x\in\mathbf{Q}$, where $\widetilde{C}_2(\omega)=\max\{\widetilde{C}_{2,i}(\omega),\,i=1,\ldots,d\}$. Let $\omega\in\Omega_2$ and let $u$ be a solution to \eqref{eq:DavieSheetGronwalld} satisfying \eqref{eq:GronwallInitialVd}. Choose $n\in\N^{\ast}$ such that $\widetilde{C}_2(\omega)\sqrt{dn}2^{-n}\leq1/2$ and split the set $[0,1]\times[0,1]$ onto $4^n$ squares $I_{n,k}\times I_{n,\kp}$. 
	Since $u\preceq u^{+}$ and for $i\in\{1,\cdots,d\},\,b_i$, is componentwise nondecreasing,  we deduce from \eqref{eq:DavieSheetGronwalld} that for every $(s,t)\in I_{n,k}\times I_{n,\kp}$ and every $i\in\{1,\cdots,d\}$:
\begin{align*}
		&u_i(s,t)-u_i(s,\kp2^{-n})-u_i(k2^{-n},t)+u_i(2^{-n}(k,\kp))\\
		&=\int_{k2^{-n}}^{s}\int_{\kp2^{-n}}^{t}\{b_i(\xi,\zeta,W_{\xi,\zeta}(\omega)+u(\xi,\zeta))-b_i(\xi,\zeta,W_{\xi,\zeta}(\omega))\}\mathrm{d}\xi\mathrm{d}\zeta\\
		&\leq\int_{k2^{-n}}^{s}\int_{\kp2^{-n}}^{t}\{b_i(\xi,\zeta,W_{\xi,\zeta}(\omega)+u^{+}(\xi,\zeta))-b_i(\xi,\zeta,W_{\xi,\zeta}(\omega))\}\mathrm{d}\xi\mathrm{d}\zeta.
	\end{align*}
	Then, using the fact that $\max\{0,x+y\}\leq\max\{0,x\}+\max\{0,y\}$ and using once more 2. in Hypothesis \ref{hyp1}, we obtain
\begin{align*}
		u_i^{+}(s,t)&\leq \max\{0,u_i(s,\kp2^{-n})+u_i(k2^{-n},t)-u_i(2^{-n}(k,\kp))\}\\
		&\qquad+\int_{k2^{-n}}^{s}\int_{\kp2^{-n}}^{t}\{b_i(\xi,\zeta,W_{\xi,\zeta}(\omega)+u^{+}(\xi,\zeta))-b_i(\xi,\zeta,W_{\xi,\zeta}(\omega))\}\mathrm{d}\xi\mathrm{d}\zeta.
	\end{align*}
	As a consequence,
	\begin{align}\label{eq:SheetEstuPlusd}
		u_i^{+}(s,t)\leq u_i^{+}(s,\kp2^{-n})+u_i^{+}(k2^{-n},t)+u_i^{-}(2^{-n}(k,\kp))
		+\varrho^{(i)}_{nk\kp}\Big(0,\oun(k+1,\kp+1)\Big)
	\end{align}
	for all $(s,t)\in I_{n,k}\times I_{n,\kp}$, where $\oun=\Big(\oun_1,\ldots,\oun_d\Big)$.
	
	Similarly, it can be shown that
	\begin{align}\label{eq:SheetEstMinusd}
		u_i^{-}(s,t)\leq u_i^{-}(s,\kp2^{-n})+u_i^{-}(k2^{-n},t)+u_i^{+}(2^{-n}(k,\kp))
		-\varrho^{(i)}_{nk\kp}\Big(0,-\uun(k+1,\kp+1)\Big)
	\end{align}
	for all $(s,t)\in I_{n,k}\times I_{n,\kp}$, where $\uun=\Big(\uun_1,\ldots,\uun_d\Big)$. We also have the following claim:
	
	\textbf{Claim}: For all $k\in\{1,2,\cdots,2^n\}$, we have
	\begin{align}\label{eq:Sheet1DEstOun}
		\max\Big\{|\oun(k,1)|,|\oun(1,k)|\Big\}\leq 3^kd^{k/2}\Big(1+2\widetilde{C}_{2}(\omega)\sqrt{dn}2^{-n}\Big)^{k+1}\beta(n)
	\end{align}
	and
	\begin{align}\label{eq:Sheet1DEstUun}
		\max\Big\{|\uun(k,1)|,|\uun(1,k)|\Big\}\leq 3^kd^{k/2}\Big(1+2\widetilde{C}_{2}(\omega)\sqrt{dn}2^{-n}\Big)^{k+1}\beta(n),
	\end{align}
	where $|\cdot|$ denotes the usual norm in $\R^d$
	\begin{proof}[Proof of the \textbf{Claim}]
		We will only prove \eqref{eq:Sheet1DEstOun} by induction and the proof of \eqref{eq:Sheet1DEstUun} will follow analogously. Let $(s,t)\in I_{n,0}\times I_{n,0}$. Since $\max\{|u|(0,0),|u|(s,0),|u|(0,t)\}\leq\beta(n)$, using \eqref{eq:SheetEstuPlusd} and \eqref{eq:coPseudoMetricc3d}, we have
		\begin{align*}
			u_i^{+}(s,t)\leq3\beta(n)+\varrho^{(i)}_{n11}\Big(0,\oun(1,1)\Big)\leq3\beta(n)+ \widetilde{C}_{2}(\omega)\sqrt{n}2^{-n}\Big(|\oun(1,1)|+\beta(n)\Big).
		\end{align*}
		By the definition of $\oun$ and the Euclidean norm
		\begin{align*}
			|\oun(1,1)|\leq 3\sqrt{d}\beta(n)+\widetilde{C}_{2}(\omega)\sqrt{dn}2^{-n}\left(|\oun(1,1)|+\beta(n)\right).
		\end{align*}
		Then, since $(1-\widetilde{C}_{2}(\omega)\sqrt{dn}2^{-n})^{-1}\leq 1+2\widetilde{C}_{2}(\omega)\sqrt{dn}2^{-n}$, we have
		\begin{align}
			|\oun(1,1)|\leq& \Big(3\sqrt{d}+\widetilde{C}_{2}(\omega)\sqrt{dn}2^{-n}\Big)\Big(1+2\widetilde{C}_{2}(\omega)\sqrt{dn}2^{-n}\Big)\beta(n)\notag\\
			\leq&3\sqrt{d}\Big(1+2\widetilde{C}_{2}(\omega)\sqrt{dn}2^{-n}\Big)^2\beta(n).
		\end{align}
		Now, let $k\in\{1,\ldots,2^n-1\}$ and by induction suppose 
		\begin{align}\label{eq:SheetInductionHypd}
			\max\Big\{|\oun(k,1)|,|\oun(1,k)|\Big\}\leq 3^kd^{k/2}\Big(1+2\widetilde{C}_{2}(\omega)\sqrt{dn}2^{-n}\Big)^{k+1}\beta(n).
		\end{align}
		Let $(s,t)\in I_{n,k}\times I_{n,0}$. Since $\max\{|u|(s,0),|u|(q2^{-n},0)\}\leq\beta(n)$, and using once more \eqref{eq:SheetEstuPlusd}, we have
		\begin{align*}
			u_i^{+}(s,t)&\leq 2\beta(n)+u_i^{+}(k2^{-n},t)+\varrho^{(i)}_{nk1}(0,\oun(k+1,1))\\
			&\leq 2\beta(n)+\oun_i(k,1)+\widetilde{C}_{2}(\omega)\sqrt{n}2^{-n}(|\oun(k+1,1)|+\beta(n))\text{ for every }i.
		\end{align*}
		Hence
		\begin{align*}
			\oun_i(k+1,1)\leq 2\beta(n)+\oun_i(k,1)+\widetilde{C}_{2}(\omega)\sqrt{n}2^{-n}(|\oun(k+1,1)|+\beta(n)) \text{ for every }i.
		\end{align*}
		As a consequence,
		\begin{align*}
			|\oun(k+1,1)|\leq2\sqrt{d}\beta(n)+|\oun(k,1)|+\widetilde{C}_{2}(\omega)\sqrt{dn}2^{-n}(|\oun(k+1,1)|+\beta(n)),
		\end{align*}
		yielding to
		\begin{align*}
			|\oun(k+1,1)|\leq \left(1+2\widetilde{C}_{2}(\omega)\sqrt{dn}2^{-n}\right)\left(2\sqrt{d}\beta(n)+|\oun(k,1)|+\widetilde{C}_{2}(\omega)\sqrt{dn}2^{-n}\beta(n)\right).
		\end{align*}
		Using $\widetilde{C}_{2}(\omega)\sqrt{n}2^{-n}\leq1$, we deduce from \eqref{eq:SheetInductionHypd} that
		\begin{align}\label{eq:SheetInductionHypS1Ad}
			|\oun(k+1,1)|&\leq\Big(1+2\widetilde{C}_{2}(\omega)\sqrt{dn}2^{-n}\Big)\Big[3^kd^{k/2}(1+2\widetilde{C}_{2}(\omega)\sqrt{dn}2^{-n})^{k+1}+3\sqrt{d}\Big]\beta(n)\nonumber\\
			&\leq \Big(3^kd^{k/2}+3\sqrt{d}\Big)\Big(1+2\widetilde{C}_{2}(\omega)\sqrt{dn}2^{-n}\Big)^{k+2}\beta(n)\nonumber\\
			&\leq3^{k+1}d^{(k+1)/2}\Big(1+2\widetilde{C}_{2}(\omega)\sqrt{dn}2^{-n}\Big)^{k+2}\beta(n).
		\end{align}
		It can also be shown analogously that
		\begin{align}\label{eq:SheetInductionHypS1Bd}
			|\oun(1,k+1)|\leq 3^{k+1}d^{(k+1)/2}\Big(1+2\widetilde{C}_{2}(\omega)\sqrt{dn}2^{-n}\Big)^{k+2}\beta(n).
		\end{align}
		This ends the proof of \eqref{eq:Sheet1DEstOun}. 
	\end{proof}

	Now we prove by induction that $k,\kp\in\{1,\ldots,2^n\}$,
	\begin{align*} 
		\max\left\{|\oun(k,\kp)|,|\uun(k,\kp)|\right\}\leq \left(3\sqrt{d}\right)^{k+\kp-1}\left(1+2\widetilde{C}_{2}(\omega)\sqrt{dn}2^{-n}\right)^{k+\kp}\beta(n).
	\end{align*}
	We deduce from \eqref{eq:Sheet1DEstOun} and \eqref{eq:Sheet1DEstUun} that \eqref{eq:GronwallSheetEstd} holds for all couples $(1,k)$, $(k,1)$, $k\in\{1,2,\ldots,2^n\}$. Fix $(k,\kp)\in\{1,2,\ldots,2^n\}$ and suppose  \eqref{eq:GronwallSheetEstd} holds for $(k,\kp)$, $(k+1,\kp)$ and $(k,\kp+1)$. It follows from (\ref{eq:SheetEstuPlusd}) that for every $(s,t)\in I_{n,k}\times I_{n,\kp}$ and every $i\in\{1,\ldots,d\}$,
	\begin{align*}
		u_i^{+}(s,t)\leq u_i^{+}(2^{-n}k,t)+u_i^{+}(s,2^{-n}\kp)+u_i^{-}(2^{-n}k,2^{-n}\kp)+\varrho^{(i)}_{nk\kp}\left(0,\oun(k+1,\kp+1)\right) 
	\end{align*}
	and by \eqref{eq:coPseudoMetricc3d},
	\begin{align*}
		&\oun_i(k+1,\kp+1)\\
		&\leq \oun_i(k,\kp+1)+\oun_i(k+1,\kp)+\uun_i(k,\kp)+\widetilde{C}_{2}(\omega)\sqrt{n}2^{-n}\Big(|\oun(k+1,\kp+1)|+\beta(n)\Big).
	\end{align*}
	Hence
	\begin{align*}
		&|\oun(k+1,\kp+1)|\\
		&\leq |\oun(k,\kp+1)|+|\oun(k+1,\kp)|+|\uun(k,\kp)|+\widetilde{C}_{2}(\omega)\sqrt{dn}2^{-n}\Big(|\oun(k+1,\kp+1)|+\beta(n)\Big),
	\end{align*}
	that is 
	\begin{align*}
		&(1-\widetilde{C}_{2}(\omega)\sqrt{dn}2^{-n})|\oun(k+1,\kp+1)|\\
		&\leq |\oun(k,\kp+1)|+|\oun(k+1,\kp)|+|\uun(k,\kp)|+\widetilde{C}_{2}(\omega)\sqrt{dn}2^{-n}\beta(n).
	\end{align*}
	Since \eqref{eq:GronwallSheetEstd} holds for $(k,\kp)$, $(k+1,\kp)$ and $(k,\kp+1)$, we obtain
	\begin{align*}
		&(1-\widetilde{C}_{2}(\omega)\sqrt{dn}2^{-n})|\oun(k+1,\kp+1)|\\
		\leq&\Big\{2(3\sqrt{d})^{k+\kp}\Big(1+2\widetilde{C}_{2}(\omega)\sqrt{dn}2^{-n}\Big)^{k+\kp+1}\\
		&+(3\sqrt{d})^{k+\kp-1}\Big(1+2\widetilde{C}_{2}(\omega)\sqrt{dn}2^{-n}\Big)^{k+\kp}+\widetilde{C}_{2}(\omega)\sqrt{dn}2^{-n}\Big\}\beta(n)\\
		\leq&\left(1+2\widetilde{C}_{2}(\omega)\sqrt{dn}2^{-n}\right)^{k+\kp+1}\Big(2(3\sqrt{d})^{k+\kp}+(3\sqrt{d})^{k+\kp-1}+1\Big)\beta(n).
	\end{align*}
	Using the inequalities $$(1-\widetilde{C}_{2}(\omega)\sqrt{dn}2^{-n})^{-1}\leq(1+2\widetilde{C}_{2}(\omega)\sqrt{dn}2^{-n})$$ and $$2(3\sqrt{d})^{k+\kp}+(3\sqrt{d})^{k+\kp-1}+1\leq (3\sqrt{d})^{k+\kp+1},$$ we have
	\begin{align*}
		|\oun(k+1,\kp+1)|\leq\left(3\sqrt{d}\right)^{k+\kp+1}\left(1+2\widetilde{C}_{2}(\omega)\sqrt{dn}2^{-n}\right)^{k+\kp+2}\beta(n).
	\end{align*}
	Similarly, we show that
	\begin{align*}
		|\uun(k+1,\kp+1)|\leq\left(3\sqrt{d}\right)^{k+\kp+1}\left(1+2\widetilde{C}_{2}(\omega)\sqrt{dn}2^{-n}\right)^{k+\kp+2}\beta(n).
	\end{align*}
The proof is completed by choosing $C_1=2\widetilde{C}_2$.
\end{proof}
\appendix
\section{Appendix}
In this section we provide a weak existense result for SDEs driven by Brownian sheet under the linear growth condition 
\begin{theorem}\label{theqweak1}
	Suppose there exists $M>0$ such that
	\begin{align*}
		|b(s,t,x)|\leq M(1+|x|),\quad\forall\,(s,t,x)\in[0,1]^2\times\R^d.
	\end{align*}
	Then \eqref{eqmainre1b} has a weak solution.
\end{theorem}
The above result is a direct consequence of the Cameron-Martin-Girsanov theorem for two-parameter processes (see e.g. \cite[Theorem 3.5]{DM15}, \cite[Proposition 1.6]{NP94}). Indeed if $X=(X_{s,t},(s,t)\in[0,1]^2)$ is a $d$-dimensional Brownian sheet given on a filtered probability space $(\Omega,\mathcal{F},(\mathcal{F}_{s,t},(s,t)\in [0,1]^2),\Pb)$, then, using Doob-Cairoli-Imkeller's maximal inequalities for two-parameter martingales (see e.g. \cite[Theorem 1]{Ca70}, \cite[Chapter II, Section 8, Theorem 1]{Im88}), one can show that the process $(\mathcal{Z}_t,t\in[0,1])$ defined by

\begin{align*}
	\mathcal{Z}_t=\exp\Big(\int_0^1\int_0^tb(s,\zeta,X_{s,\zeta}) \mathrm{d}X_{s,\zeta}-\dfrac{1}{2}\int_0^1\int_0^t|b(s,\zeta,X_{s,\zeta})|^2\mathrm{d}s \mathrm{d}\zeta\Big)
\end{align*}
is a martingale with respect to {the filtration $(\mathcal{F}_{1,t},t\in[0,1])$}. Then, by \cite[Theorem 3.5]{DM15}, the process $(W_{s,t},(s,t)\in[0,1]^2)$ given by
\begin{align*}
	W_{s,t}=X_{s,t}-X_{s,0}-X_{0,t}+X_{0,0}-\int_0^s\int_0^tb(\xi,\zeta,X_{\xi,\zeta})\mathrm{d}\xi \mathrm{d}\zeta,\quad\forall\,(s,t)\in[0,1]^2,
\end{align*}
is a $\R^d$-valued $(\mathcal{F}_{s,t})$-Brownian sheet with $\partial W=0$ under the probability $\widetilde{\Pb}$ defined by 
\begin{align*}
	\frac{\mathrm{d}\widetilde{\Pb}}{\mathrm{d}\Pb}=\mathcal{Z}_1.
\end{align*}

\section*{Ackowledgment}

The authors wish to thank an anonymous referee and the editor for their valuable comments and suggestions. 



\begin{thebibliography}{10}
	
	\bibitem{ABP17a}
	O.~Amine, David Ba{\~n}os, and Frank Proske.
	\newblock {$C^{\infty}$} regularization by noise of singular {ODE's}.
	\newblock {\em arXiv preprint arXiv:1710.05760}, 2017.
	
	\bibitem{AMP20}
	O.~Amine, A.-R. Mansouri, and F.~Proske.
	\newblock Well-posedness of the deterministic transport equation with singular
	velocity field perturbed along fractional brownian paths.
	\newblock {\em arXiv preprint arXiv:2003.06200}, 2020.
	
	\bibitem{BY82}
	M.~T. Barlow and M.~Yor.
	\newblock Semi-martingale inequalities via the garsia-rodemich-rumsey lemma,
	and applications to local times.
	\newblock {\em Journal of Functional Analysis}, 49(2):198--229, 1982.
	
	\bibitem{BFGM14}
	L.~Beck, F.~Flandoli, M.~Gubinelli, and M.~Maurelli.
	\newblock Stochastic {ODE}s and stochastic linear {PDE}s with critical drift:
	regularity, duality and uniqueness.
	\newblock {\em Electronic Journal of Probability}, 24, 2019.
	
	\bibitem{BDM21}
	A.-M. Bogso, M.~Dieye, and O.~Menoukeu~Pamen.
	\newblock Stochastic integration with respect to local time of the brownian
	sheet and regularising properties of brownian sheet paths.
	\newblock {\em arXiv preprint arXiv:2112.00401}, 2021.
	
	\bibitem{BM16}
	O.~Butkovsky and L.~Mytnik.
	\newblock Regularization by noise and flows of solutions for a stochastic heat
	equation.
	\newblock {\em Annals of Probability}, 47:169--212, 2019.
	
	\bibitem{Ca70}
	R.~Cairoli.
	\newblock Une in{\'e}galit{\'e} pour martingales {\`a} indices multiples et ses
	applications.
	\newblock {\em S{\'e}minaire de probabilit{\'e}s de Strasbourg}, 4:1--27, 1970.
	
	\bibitem{Ca72}
	R.~Cairoli.
	\newblock Sur une {\'e}quation diff{\'e}rentielle stochastique.
	\newblock {\em C.R. Acad. Sci. Paris S{\'e}rie A}, 274:1739--1742, 1972.
	
	\bibitem{CK91}
	E.~Carlen and P.~Kree.
	\newblock Lp estimates on iterated stochastic integrals.
	\newblock {\em The Annals of Probability}, pages 354--368, 1991.
	
	\bibitem{CaNu88}
	R.~Carmona and D.~Nualart.
	\newblock Random non-linear wave equations: smoothness of the solutions.
	\newblock {\em Probability Theory and Related Fields}, 79(4):469--508, 1988.
	
	\bibitem{CG16}
	R.~Catellier and M.~Gubinelli.
	\newblock Averaging along irregular curves and regularisation of {ODE}s.
	\newblock {\em Stochastic Process. Appl.}, 126:2323--2366, 2016.
	
	\bibitem{DM15}
	R.~C. Dalang and C.~Mueller.
	\newblock Multiple points of the brownian sheet in critical dimensions.
	\newblock {\em The Annals of Probability}, 43(4):1577--1593, 2015.
	
	\bibitem{Da07}
	A.~M. Davie.
	\newblock Uniqueness of solutions of stochastic differential equations.
	\newblock {\em International Mathematics Research Notices}, Vol. 2007, 2007.
	
	\bibitem{DavieM}
	A.~M. Davie.
	\newblock Individual path uniqueness of solutions of stochastic differential
	equations.
	\newblock In {\em In Stochastic Analysis}, pages 213--225. Springer-Heidelberg,
	2011.
	
	\bibitem{De70}
	K.~Deimling.
	\newblock A carath{\'e}odory theory for systems of integral equations.
	\newblock {\em Annali di Matematica Pura ed Applicata}, 86(1):217--260, 1970.
	
	\bibitem{Ei00}
	N.~Eisenbaum.
	\newblock Integration with respect to local time.
	\newblock {\em Potential analysis}, 13(4):303--328, 2000.
	
	\bibitem{ENO03}
	M.~Erraoui, Y.~Ouknine, and D.~Nualart.
	\newblock Hyperbolic stochastic partial differential equations with additive
	fractional brownian sheet.
	\newblock {\em Stochastics and Dynamics}, 3(02):121--139, 2003.
	
	\bibitem{FaNu93}
	M.~Farr{\'e} and D.~Nualart.
	\newblock Nonlinear stochastic integral equations in the plane.
	\newblock {\em Stochastic processes and their applications}, 46(2):219--239,
	1993.
	
	\bibitem{Fla10}
	F.~Flandoli.
	\newblock {\em Random Perturbation of PDE's and Fluid Dynamic Models: {\'E}cole
		d'{\'e}t{\'e} de Probabilit{\'e}s de Saint-Flour XL 2010}, volume 2015.
	\newblock Springer Science \& Business Media, 2010.
	
	\bibitem{GG21}
	L.~Galeati and M.~Gubinelli.
	\newblock Noiseless regularisation by noise.
	\newblock {\em Revista Matem{\'a}tica Iberoamericana}, 38(2):433--502, 2021.
	
	\bibitem{HP21}
	F.~A. Harang and N.~Perkowski.
	\newblock C$^{\infty}$ regularization of odes perturbed by noise.
	\newblock {\em Stochastics and Dynamics}, 21(08):2140010, 2021.
	
	\bibitem{Im88}
	P.~Imkeller.
	\newblock {\em Two-parameter martingales and their quadratic variation}, volume
	1308.
	\newblock Springer, 1988.
	
	\bibitem{KP20a}
	H.~Kremp and N.~Perkowski.
	\newblock Multidimensional {SDE} with distributional drift and l{\'e}vy noise.
	\newblock {\em Bernoulli}, 28(3):1757--1783, 2022.
	
	\bibitem{Nu06}
	D.~Nualart.
	\newblock {\em The Malliavin calculus and related topics}, volume 1995.
	\newblock Springer, 2006.
	
	\bibitem{NP94}
	D.~Nualart and E.~Pardoux.
	\newblock Markov field properties of solutions of white noise driven
	quasi-linear parabolic pdes.
	\newblock {\em Stochastics: An International Journal of Probability and
		Stochastic Processes}, 48(1-2):17--44, 1994.
	
	\bibitem{NuTi97}
	D.~Nualart and S.~Tindel.
	\newblock Quasilinear stochastic hyperbolic differential equations with
	nondecreasing coefficient.
	\newblock {\em Potential Analysis}, 7(3):661--680, 1997.
	
	\bibitem{NuTi98}
	D.~Nualart and S.~Tindel.
	\newblock On two-parameter non-degenerate brownian martingales.
	\newblock {\em Bulletin des sciences mathematiques}, 122(4):317--335, 1998.
	
	\bibitem{NuYe89}
	D.~Nualart and J.~Yeh.
	\newblock Existence and uniqueness of a strong solution to stochastic
	differential equations in the plane with stochastic boundary process.
	\newblock {\em Journal of Multivariate Analysis}, 28(1):149--171, 1989.
	
	\bibitem{Pri18}
	E.~Priola.
	\newblock Davie's type uniqueness for a class of {SDE}s with jumps.
	\newblock {\em Ann. Inst Henri Poincar\'e: Prob. Stat}, 54:694--725, 2018.
	
	\bibitem{Pri19}
	E.~Priola.
	\newblock On davie's uniqueness for some degenerate {SDE}s.
	\newblock arXiv:1912.02776, 2019.
	
	\bibitem{QuTi07}
	L.~Quer-Sardanyons and S.~Tindel.
	\newblock The 1-d stochastic wave equation driven by a fractional brownian
	sheet.
	\newblock {\em Stochastic processes and their applications},
	117(10):1448--1472, 2007.
	
	\bibitem{Ra76}
	D.~L. Rasmussen.
	\newblock Gronwall{'s} inequality for functions of two independent variables.
	\newblock {\em Journal of Mathematical Analysis and Applications},
	55(2):407--417, 1976.
	
	\bibitem{Sh16}
	A.~V. Shaposhnikov.
	\newblock Some remarks on davi{e}'s uniqueness theorem.
	\newblock {\em Proceedings of the Edinburgh Mathematical Society},
	59(4):1019--1035, 2016.
	
	\bibitem{ShWr20}
	Alexander Shaposhnikov and Lukas Wresch.
	\newblock Pathwise vs. path-by-path uniqueness.
	\newblock {\em arXiv preprint arXiv:2001.02869}, 2020.
	
	\bibitem{Sno72}
	D.~R. Snow.
	\newblock Gronwall{'s} inequality for systems of partial differential equations
	in two independent variables.
	\newblock {\em Proceedings of the American Mathematical Society}, 33(1):46--54,
	1972.
	
	\bibitem{Tu83}
	C.~Tudor.
	\newblock On the two parameter {It\^o} equations.
	\newblock {\em Publications math{\'e}matiques et informatique de Rennes},
	(1):1--17, 1983.
	
	\bibitem{Wa78}
	J.~B. Walsh.
	\newblock The local time of the brownian sheet.
	\newblock {\em Ast{\'e}risque}, 52(53):47--61, 1978.
	
	\bibitem{CW75}
	R~Walsh and Cairoli~J. B.
	\newblock Stochastic integrals in the plane.
	\newblock {\em Acta Math}, 134:111--183, 1975.
	
	\bibitem{Ye81}
	J.~Yeh.
	\newblock Existence of strong solutions for stochastic differential equations
	in the plane.
	\newblock {\em Pacific Journal of Mathematics}, 97(1):217--247, 1981.
	
	\bibitem{Ye85}
	J.~Yeh.
	\newblock Existence of weak solutions to stochastic differential equations in
	the plane with continuous coefficients.
	\newblock {\em Transactions of the American Mathematical Society},
	290(1):345--361, 1985.
	
	\bibitem{Ye87}
	J.~Yeh.
	\newblock Uniqueness of strong solutions to stochastic differential equations
	in the plane with deterministic boundary process.
	\newblock {\em Pacific journal of mathematics}, 128(2):391--400, 1987.
	
\end{thebibliography}

\end{document}